\documentclass[11pt,a4j,twoside]{article}

\usepackage{amsmath,amssymb}
\usepackage{amsthm}
\usepackage[abbrev]{amsrefs}
\usepackage{latexsym}
\usepackage{graphicx}

\usepackage[english]{babel}
\usepackage{fancyhdr}
\usepackage{a4wide}

\usepackage{bm}

\usepackage[pagewise, mathlines]{lineno}

\allowdisplaybreaks[1]

\numberwithin{equation}{section}

\newtheorem{thm}{Theorem}[section]
\newtheorem{prop}[thm]{Proposition}
\newtheorem{lem}[thm]{Lemma}
\newtheorem{cor}[thm]{Corollary}
\newtheorem{dfn}[thm]{Definition}
\theoremstyle{definition} 
\newtheorem{ex}[thm]{Example}
\newtheorem{rem}[thm]{Remark}

\newcommand\ND{\newcommand}

\ND\lref[1]{Lemma~\ref{#1}}
\ND\tref[1]{Theorem~\ref{#1}}
\ND\pref[1]{Proposition~\ref{#1}}
\ND\sref[1]{Section~\ref{#1}}
\ND\ssref[1]{Subsection~\ref{#1}}
\ND\aref[1]{Appendix~\ref{#1}}
\ND\rref[1]{Remark~\ref{#1}}
\ND\cref[1]{Corollary~\ref{#1}}
\ND\eref[1]{Example~\ref{#1}}
\ND\fref[1]{Fig.\ {#1} }
\ND\lsref[1]{Lemmas~\ref{#1}}
\ND\tsref[1]{Theorems~\ref{#1}}
\ND\dref[1]{Definition~\ref{#1}}
\ND\psref[1]{Propositions~\ref{#1}}
\ND\rsref[1]{Remarks~\ref{#1}}
\ND\sssref[1]{Subsections~\ref{#1}}
\ND\esref[1]{Examples~\ref{#1}}
\ND\asref[1]{Assumption~\ref{#1}}

\newcommand{\ep}{\varepsilon}

\newcommand{\wilde}{\widetilde}
\newcommand{\h}{\quad}
\newcommand{\dis}{\displaystyle}

\newcommand{\cl}{c\`adl\`ag\ }
\newcommand{\Prob}{\mathbb{P}}
\newcommand{\Ex}{\mathbb{E}}
\newcommand{\bi}{\bar{\iota}}
\newcommand{\II}{\mathrm{I\hspace{-1.2pt}I}}

\newcommand{\bv}{\mathbf{v}}
\newcommand{\bu}{\mathbf{u}}
\newcommand{\bU}{\mathbf{U}}
\newcommand{\bV}{\mathbf{V}}
\newcommand{\bW}{\mathbf{W}}
\newcommand{\bh}{\mathbf{h}}

\newcommand{\bdf}{\mathbf{f}}
\newcommand{\ba}{\mathbf{a}}
\newcommand{\bb}{\mathbf{b}}
\newcommand{\bcl}{\mathbf{cl}}
\newcommand{\bvl}{\mathbf{vl}}

\newcommand{\be}{\mathbf{e}}
\newcommand{\bB}{\mathbf{B}}
\newcommand{\bA}{\mathbf{A}}
\newcommand{\bH}{\mathbf{H}}
\newcommand{\bJ}{\mathbf{J}}
\newcommand{\bdm}{\mathbf{m}}
\newcommand{\cA}{\mathcal{A}}
\newcommand{\cC}{\mathcal{C}}
\newcommand{\cD}{\mathcal{D}}
\newcommand{\cE}{\mathcal{E}}
\newcommand{\cF}{\mathcal{F}}
\newcommand{\cH}{\mathcal{H}}

\newcommand{\cK}{\mathcal{K}}
\newcommand{\cL}{\mathcal{L}}

\newcommand{\cO}{\mathcal{O}}
\newcommand{\cR}{\mathcal{R}}
\newcommand{\cV}{\mathcal{V}}
\newcommand{\cW}{\mathcal{W}}
\newcommand{\cX}{\mathcal{X}}
\newcommand{\cY}{\mathcal{Y}}
\newcommand{\cZ}{\mathcal{Z}}

\newcommand{\bbR}{\mathbb{R}}
\newcommand{\sJ}{\mathsf{J}}
\newcommand{\sk}{\mathsf{k}}

\title{Stochastic Jacobi fields along discontinuous martingales on Riemannian submanifolds}
\date{}
\author{Fumiya Okazaki}
\begin{document}
\maketitle
\footnote{}
\footnote{Email address: okazaki.f.660b@m.isct.ac.jp}
\renewcommand{\thepage}{\arabic{page}}
\begin{abstract}
In this article, we consider discontinuous martingales on tangent bundles over submanifolds of Euclidean space. First, we introduce a connection rule on tangent bundles and establish the It\^o calculus for discontinuous semimartingales on tangent bundles. Then we focus on harmonic maps with respect to non-local Dirichlet forms and show that the derivative of harmonic maps along infinitesimal symmetries induces discontinuous martingales on tangent bundles. This process may be viewed as a stochastic Jacobi field along the image martingale.
We also introduce the stochastic parallel transport of tangent vectors along \cl semimartingales on Riemannian submanifolds with projected jumps. Using the parallel transport, we obtain the mean-value property for the differential of harmonic maps involving a jump part expressed through the second fundamental form. We also obtain a derivative formula for harmonic maps for isotropic L\'evy processes with a Brownian component on compact Riemannian manifolds.
\end{abstract}

\begin{flushleft}
{\bf MSC2020 Subject Classifications: 58J65, 60J46}
\end{flushleft}

\section{Introduction and main results}
Harmonic maps are defined by critical points of the energy functional for maps between two Riemannian manifolds. Let $(E,h)$, $(M,g)$ be Riemannian manifolds and assume that $M$ is isometrically embedded in the higher-dimensional Euclidean space $\mathbb{R}^d$. For a smooth map $u \colon E \to M$, we define the energy of $u$ by
\[
\cE(u):=\frac{1}{2}\int_E | \nabla^E u |^2(z) \, \bdm(dz),
\]
where $\bdm$ is the Riemannian volume measure for the metric $h$, $\nabla^E$ is the gradient operator on $E$, and we regard $u$ as an $\mathbb{R}^d$-valued map. A map $u$ is called a harmonic map if it is a critical point of the energy which is characterized by the Euler-Lagrange equation
\[
\Delta_E u(z) \perp T_{u(z)}M,\ z\in E,
\]
where $\Delta_E$ is the Laplace-Beltrami operator on $(E,h)$ and we regard $u$ as an $\mathbb{R}^d$-valued function.
On the other hand, probabilistic approaches for harmonic maps between Riemannian manifolds have been developed in relation to Brownian motion and martingales on Riemannian manifolds. Indeed, by \cites{Meyer81}, it has been shown that the harmonicity of a smooth map can be characterized through the property that the image of Brownian motion on $(E,h)$ is a continuous martingale on $(M,g)$. In \cite{Pic01}, this characterization is established for harmonic maps in a weak sense with respect to strongly local Dirichlet forms on locally compact separable metric spaces in terms of their associated diffusions.

One way to investigate how the geometry of manifolds affects harmonic maps is to study the differentials of harmonic maps.
A probabilistic approach to the differentials of harmonic maps based on Jacobi fields along continuous martingales on manifolds was developed in \cites{Arnaudon96, AT98, AT98_2}. For a given smooth harmonic map between two Riemannian manifolds, by regarding the initial value of Brownian motion as another parameter of the stochastic process, we can obtain a family of continuous martingales on the target manifold which smoothly depends on the initial values of Brownian motion. Then differentiating the family of martingales with respect to the initial values of Brownian motion yields a stochastic process on $TM$. It has been shown in \cite{Arnaudon96} that the derivatives of differentiable families of continuous martingales are martingales on tangent bundles. Martingales on tangent bundles have also been employed in the study of the regularity of harmonic maps and their gradients \cites{AT98_2, ALT99, Pic00}. A key feature of these martingales on manifolds is that they reflect geometric properties of manifolds such as curvature. For example, the Jacobi field associated with the variation of Brownian motion on a Riemannian manifold can be described through the Ricci curvature and it is related to the gradient formula for solutions of the heat equation on the Riemannian manifold.

In this article, we focus on harmonic maps with respect to non-local energy functionals. A typical example of the non-local energy is a fractional Dirichlet energy, which is the energy associated with a fractional Laplacian. 
In a similar way to the classical harmonic maps with respect to the Laplacian, critical points of the fractional Dirichlet energy are called fractional harmonic maps. The notion of fractional harmonic maps was introduced by \cites{DaLioRiv11, DaLioRiv112} and their regularity has been studied in \cites{DaLioRiv11, DaLioRiv112, MazoSchik18, MP20, MPS21, MSire15}.

Our interest in this article lies in the relationship between these harmonic maps and stochastic processes. Recently, harmonic maps valued in Riemannian submanifolds of the higher-dimensional Euclidean space with respect to non-local Dirichlet forms on locally compact separable metric spaces have been characterized through stochastic processes in \cite{Oka24} by employing the notion of discontinuous martingales on manifolds, which was introduced in \cite{Pic91} and further studied in \cites{Pic94, Oka23}.

To investigate how harmonic maps with respect to non-local Dirichlet forms reflect the geometry of manifolds, we impose additional regularity on a harmonic map $u$ and study the stochastic process induced by the derivative of $u$. In this article, instead of directly differentiating a family of martingales $\{u(Z^z)\}_{z \in E}$ with respect to $z$, where $Z^z$ is the Markov process associated with a Dirichlet form starting from $z \in E$, we consider the derivatives with respect to vector fields $\cH$ which generates the infinitesimal symmetries of the Dirichlet form. The stochastic process $\cH u(Z^z_t)$ can formally be regarded as the differential
\[
\left(\frac{d}{d\ep}\right)_{\ep=0}u(\Phi_{\ep}(Z^z_t)),
\]
where $\Phi_{\ep}$ is the flow generated by $\cH$. In this sense, the differential process may be regarded as a stochastic Jacobi field along the image martingale.
Under the assumption for symmetries of the Dirichlet form with respect to $\cH$ described in \thetag{C} later, it is natural to expect that the differential process $\cH u(Z)$ is also a martingale on $TM$ in some sense and it reflects the geometry of $M$. In order to verify that statement, we introduce the notion of discontinuous martingales on tangent bundles. In the continuous case, we need a linear connection on $M$ to define the It\^o integral of 1-forms along continuous semimartingales on $M$. Then continuous martingales on $M$ are defined as continuous semimartingales for which the It\^o integral of each 1-form is a local martingale.
However, as we can see in \cites{Pic91, Pic94, Oka23, Oka24}, a linear connection alone does not naturally specify the definition of discontinuous martingales on manifolds and an additional connection rule is required.

To begin with, we briefly recall the It\^o calculus on manifolds for discontinuous semimartingales and clarify the definition of discontinuous martingales on manifolds. We fix a filtered probability space $(\Omega, \mathcal{F}, \{\mathcal{F}_t\}_{t\geq 0}, \Prob)$ satisfying the usual conditions. First, we assume that $M$ is simply a manifold. The definition of semimartingales on $M$ is simple:
\begin{dfn}
Let $X$ be an $M$-valued \cl $\{\cF_t\}_{t\geq 0}$-adapted process. The process $X$ is called an $M$-valued semimartingale if for all $f\in C^{\infty}(M)$, $f(X)$ is a semimartingale on $\mathbb{R}$.
\end{dfn}
In \cite{Pic91}, the map called connecteur, or connection rule, has been introduced in order to define the It\^o integral of 1-forms for discontinuous semimartingales on manifolds. 
\begin{dfn}
Let $\gamma:M\times M\to TM$ be a Borel measurable map and suppose $\gamma$ is $C^2$ on a neighborhood of $\mathrm{diag}(M)$, that is, there exists an open neighborhood $\mathcal{U} \subset M \times M$ of $\mathrm{diag}(M)$ such that $\gamma|_{\mathcal{U}} \in C^2(\mathcal{U}\, ;TM)$. Then $\gamma$ is called a connection rule if it satisfies the following conditions: For all $x,y\in M$,
\begin{enumerate}
\item[(i)] $\gamma (x,y)\in T_xM$;
\item[(ii)] $\gamma (x,x)=0$;
\item[(iii)] $(d \gamma (x,\cdot))_x=id_{T_xM}.$
\end{enumerate}
\end{dfn}
For a given semimartingale $X$ on $M$ and a connection rule $\gamma$, we can regard jumps of $X$ as $\gamma(X_{t-},X_t) \in T_{X_{t-}}M$ once we make a choice of a connection rule. Then for each smooth 2-tensor field $b$ on $M$, we can define the quadratic variation $\dis \int b(X_-)\, d^{\gamma} [X,X]$ (see \cite{Pic91} or \cite{Oka23} for details). Denote the continuous locally bounded variation part by $\dis \int b(X_-)\, d[X,X]^c$, which is independent of the choice of the connection rule. For an $M$-valued semimartingale $X$, the It\^o integral of $T^*M$-valued predictable processes $\alpha$ with $\alpha_t \in T^*_{X_{t-}}M$ for each $t\geq 0$ a.s., denoted by $\dis \alpha \mapsto \int \alpha \, d^{\gamma}X$, is uniquely defined as a map having the following properties:
\begin{enumerate}
\item For any $\mathbb{R}$-valued predictable process $K$,
\[
\int K \alpha\, d^{\gamma}X=\int K  \, d \left( \int \alpha \, d^{\gamma}X \right).
\]
\item For every $f\in C^2(M)$,
\begin{align*}
\int df(X_-)\, d^{\gamma} X&=f(X)-f(X_0)-\frac{1}{2}\int \nabla df (X_-)\, d[X,X]^c\\
&\h -\sum_{0<s\leq \cdot} \{f(X_s)-f(X_{s-})-\langle df(X_{s-}),\gamma(X_{s-},X_s) \rangle \}.
\end{align*}
\end{enumerate}
\begin{dfn}
Let $\gamma$ be a connection rule on $M$. An $M$-valued semimartingale $X$ is called a $\gamma$-martingale if for any $\alpha \in \Omega^1(M)$, the It\^o integral $\dis \int \alpha(X_-)\, d^{\gamma}X$ is a local martingale.
\end{dfn}
Suppose that the manifold $M$ is a properly embedded Riemannian submanifold of the higher-dimensional Euclidean space $\mathbb{R}^d$ with an isometric embedding of $\iota \colon M \to \mathbb{R}^d$. Then the embedding provides a connection rule $\gamma$ on $M$ defined by
\begin{align}\label{ConneEmbed}
\gamma (x,y) = \Pi_x(y-x),
\end{align}
where $\Pi_x \colon \mathbb{R}^d \to T_xM$ is the orthogonal projection. One significant fact regarding connection rules is that each connection rule induces a torsion-free connection. For example, the connection induced by the connection rule \eqref{ConneEmbed} is the Levi-Civita connection.

Next we determine the notion of martingales on tangent bundles. Throughout this article, we focus on the case where the base manifold $M$ is a properly embedded Riemannian submanifold of the higher-dimensional Euclidean space $\mathbb{R}^d$ and the connection rule $\gamma$ is given by \eqref{ConneEmbed}. We define the connection rule $\gamma^{\bcl} \colon TM \times TM \to TTM$ by
\begin{align*}
\gamma^{\bcl} (\bU,\bV)=(\gamma (x, y), \II(\gamma(x, y), \bU)) + (0,-A_{\Pi_x^{\perp}(y-x)}\bU+ \Pi_x(\bV-\bU))
\end{align*}
for $\bU \in T_xM, \bV \in T_yM$. Here $\Pi_x^{\perp}\colon \mathbb{R}^d \to T_x^{\perp}M$ is the orthogonal projection, $\II \in \Gamma(TM^{\perp}\otimes T^*M \otimes T^*M)$ and $A \in \Gamma(\mathrm{Hom}(TM)\otimes (TM^{\perp})^*)$ are the second fundamental form and the Weingarten map for the embedding $\iota \colon M \to \mathbb{R}^d$ defined by
\begin{align*}
\begin{cases}
D_{\cX}\cY &= \nabla_{\cX}\cY + \II(\cX,\cY),\\
g(A_{\xi}\cX, \cY) &= -\langle \xi, \II(\cX,\cY) \rangle,
\end{cases}
\end{align*}
respectively for $\cX,\cY \in \mathfrak{X}(M)$ and $\xi \in \Gamma(T^{\perp}M)$, where $D$ is the covariant derivative on $\mathbb{R}^d$ and $\Gamma(\cdot)$ stands for the set of sections of vector bundles. See \pref{cpltgamma} for the details of these notions. We will also check that $\gamma^{\bcl}$ is actually a connection rule on $TM$ in \pref{cpltgamma}. Moreover, we will show that the connection rule $\gamma^{\bcl}$ induces the complete lift of the Levi-Civita connection. This means that the notion of $\gamma^{\bcl}$-martingales is an extension of that of continuous martingales on $TM$ with respect to the complete lift of the Levi-Civita connection.\\

Next, we clarify the setting for harmonic maps with respect to non-local Dirichlet forms between Riemannian manifolds. We refer to \cites{FOT, ChenFuku} for the notions of Dirichlet forms but in this paper, we only consider Dirichlet forms on Riemannian manifolds. Let $(E,h)$ be a second-countable Riemannian manifold. We denote the Riemannian volume measure by $\bdm$. Let $(\cE,\cF)$ be a symmetric regular Dirichlet form on $L^2(E;\bdm)$, i.e. $\cF$ is a dense subspace of $L^2(E;\bdm)$, $\cE \colon \cF \times \cF \to \mathbb{R}$ is a symmetric bilinear closed Markovian form and there exists a vector subspace $\cC \subset C_0(E)\cap \cF$, which is called a core of $(\cE,\cF)$, such that $\cC$ is $\cE_1$-dense in $\cF$ and $\|\cdot \|_{\infty}$-dense in $C_0(E)$. Here the inner product $\cE_1$ on $\cF$ is defined by
\[
\cE_1(u,v)=\cE(u,v) + \langle u, v \rangle_{L^2(E;\bdm)}
\]
and denote the induced norm by $\|\cdot \|_{\cE_1}$.
The extended Dirichlet space $\cF_e$ is defined by the set of functions $u$ on $E$ defined $\bdm$-a.e. such that there exists a sequence $\{u_k\}_{k \in \mathbb{N}} \subset \cF$ with
\begin{align*}
\lim_{m,n \to \infty}\cE(u_m-u_n,u_m-u_n)=0,\ \lim_{n\to \infty}u_n = u,\ \bdm \text{-a.e.}
\end{align*}
We denote the quasi-continuous modification of a function $u\in \cF_e$ by $\tilde u$. For an open set $D\subset E$, we set
\begin{align*}
\begin{cases}
\cF^D&:=\{ u \in \cF \mid \tilde u=0\ \text{q.e.}\ \text{on}\ E \backslash D \},\\
\cE^D(u,v)&:=\cE(u,v),\ u,v \in \cF^D,
\end{cases}
\end{align*}
where "q.e." stands for "except for a zero capacity set with respect to $\cE_1$".
Then $(\cE^D, \cF^D)$ is a regular Dirichlet form on $L^2(D;\bdm)$ and its extended Dirichlet space coincides with
\[
\cF_e^D=\{ u \in \cF_e \mid \tilde u=0\ \text{q.e.}\ \text{on}\ E \backslash D \}.
\]
We also consider the localized Dirichlet spaces as follows. We say that a function $u \colon E \to \bbR$ belongs to $\cF^D_{loc}$ if for each relatively compact open set $D_1$ with $\overline{D_1} \subset D$, there exists $u_{D_1} \in \cF^D$ such that $u=u_{D_1}$ a.e. on $D_1$. Each function $u \in \cF^D_{loc}$ also has a quasi-continuous modification on $D$, which we denote by $\tilde{u}$ likewise. We further set
\begin{align*}
\cF^D_{loc}(\mathbb{R}^d)&:=\{ u=(u^1,\dots,u^d) \colon E \to \bbR^d \mid u^i \in \cF^D_{loc}\ \text{for each}\ i=1,\dots,d \},\\
\cF^D_{loc}(M)&:=\{ u\in \cF^D_{loc}(\mathbb{R}^d)\mid u(z) \in M\ \text{a.e.}\ z\in E \},\\
\cF^D_e(\mathbb{R}^d)&:=\{ \phi=(\phi^1,\dots,\phi^d) \colon E \to \bbR^d \mid \phi^i \in \cF^D_e\ \text{for each}\ i=1,\dots,d \},\\
\cF^D(u^*TM)&:=\{ \phi \in \cF^D(\mathbb{R}^d)\mid \phi(z) \in T_{u(z)}M\ \text{a.e.}\ z\in E \}.
\end{align*}
We use the following Beurling-Deny decomposition of $\cE$. For $u,v \in \cF_e$
\begin{align}\label{BDdecom}
\cE(u,v)=\frac{1}{2}\mu_{\langle u,v \rangle}^c(E)+\frac{1}{2}\int_{E \times E \backslash \mathrm{diag}(E)}(\tilde{u}(z)-\tilde{u}(w))(\tilde{v}(z)-\tilde{v}(w))\, \sJ(dzdw)+\int_E \tilde{u}\tilde{v} \, d\sk,
\end{align}
where $\mu_{\langle u,v \rangle}^c$ is the mutual energy measure of the continuous part associated with $u,v$, $\sJ$ is the symmetric jump measure on $E\times E \backslash \mathrm{diag}(E)$ and $\sk$ is the killing measure on $E$.

For a regular Dirichlet form $(\cE,\cF)$, there exists an $\bdm$-symmetric Hunt process
\[
(\Omega,\{Z_t\}_{t\geq 0},\{\theta_t\}_{t\geq0}, \zeta, \{\Prob_z\}_{z\in E_{\varDelta}})
\]
on $E$ satisfying
\[
\frac{1}{t}\langle u-p_tu, u \rangle_{L^2(E;\bdm)} \nearrow \cE(u,u)\ \text{for}\ u\in b\mathcal{B}(E)\cap \cF,
\]
where $\theta \colon \Omega \to \Omega$ is the shift operator, $\zeta$ is the lifetime of $Z$, $\varDelta$ is the cemetery point, $E_{\varDelta}=E \cup \{\varDelta\}$ and $\{p_t\}_{t\geq 0}$ is the transition function of $Z$. We set
\begin{align*}
\mathcal{F}^0_{\infty}&=\sigma(Z_s;\ s< \infty),\\
\mathcal{F}^0_t&=\sigma(Z_s;\ s\leq t).
\end{align*}
For a $\sigma$-finite measure $\mu$, we denote the $\mathbb{P}_{\mu}$-completion of $\mathcal{F}^0_{\infty}$ by $ \mathcal{F}^{\mu}_{\infty}$, where
\[
\mathbb{P}_{\mu}(\Lambda)=\int_{E_{\varDelta}}\mathbb{P}_z(\Lambda)\, \mu(dz),\ \Lambda \in \mathcal{F}^0_{\infty}.
\]
Set
\[
\mathcal{F}^{\mu}_t=\sigma(\mathcal{F}^0_t,\mathcal{N}_{\mu}),
\]
where $\mathcal{N}_{\mu}$ is the family of all $\mathbb{P}_{\mu}$-null sets in $\mathcal{F}^{\mu}_{\infty}$. Denote the set of probability measures on $E_{\varDelta}$ by $\mathcal{P}(E_{\varDelta})$ and let
\[
\mathcal{F}^Z_t=\bigcap_{\mu \in \mathcal{P}(E_{\varDelta})} \mathcal{F}^{\mu}_t,\ t\in [0,\infty].
\]
For an open set, we set the first exit time
\[
\tau_D:=\inf \{ t> 0 \mid Z_t \notin D \}.
\]
In previous articles such as \cites{Chen09, Oka24}, following conditions (A) and (B) for $u\in \cF^D_{loc}\cap L^{\infty}_{loc}(D)$ have been considered: For any relatively compact open set $D_1$, $D_2$ with
\begin{align}\label{opensets}
\overline{D_1}\subset D_2 \subset \overline{D_2} \subset D,
\end{align}
\begin{description}
\item[\thetag{A}]$u$ satisfies $\int_{D_1\times (E \backslash \, D_2)}|u(w)|\, \sJ(dzdw) <\infty$;
\item[\thetag{B}]if we define a function $f_u$ by
\begin{align}
f_u(z):= \mathbf{1}_{D_1}(z)\mathbb{E}_z\left[ ((1-\phi_{D_2})|u|)(Z_{\tau_{D_1}})\right]\ \text{for}\ z\in E,\label{fu}
\end{align}
then $f_u \in \mathcal{F}_e^{D_1}$, where $\phi_{D_2}$ is a function satisfying
\[
\phi_{D_2}\in \mathcal{F}\cap C_0(D),\ 0\leq \phi_{D_2} \leq 1,\ \phi_{D_2}=1\ \text{on}\ D_2.
\]
\end{description}
For vector-valued maps, conditions \thetag{A} and \thetag{B} are understood componentwise.
If $u\in \cF_{loc}^D\cap L^{\infty}_{loc}(D)$ satisfies (A), then for each relatively compact open set $D_1$ with $\overline{D_1} \subset D$ and $\phi \in \cF^{D_1} \cap L^{\infty}(E)$, we can determine the value $\cE(u,\phi)$ as a finite value through the Beurling-Deny decomposition of $(\cE,\cF)$. We will see this fact in detail later in \sref{proof}. Let $\cC$ be a special standard core of $(\cE,\cF)$ and set $\cC_D=\cC \cap C_0(D)$, $\dis \cC_D(\mathbb{R}^d)=\left(\cC_D \right)^d$. For $u \in \cF_{loc}^D(M) \cap L^{\infty}_{loc}(D)$, we consider the class of test functions $\Pi_u \cC_D$ defined by
\[
\Pi_u \cC_D(\mathbb{R}^d):=\{ z \mapsto \Pi_{u(z)}\phi(z) \mid \phi \in \cC_D(\mathbb{R}^d) \}.
\]
Then since we assume that $M$ is properly embedded, for each $\psi \in \Pi_u \cC_D(\mathbb{R}^d)$ and a relatively compact open set $D_1$ with $\mathrm{supp}[\psi] \subset D_1 \subset \overline{D_1} \subset D$, $\psi \in \cF^{D_1}(u^*TM)$. Based on this fact, we define harmonic maps with respect to non-local Dirichlet forms as follows.
\begin{dfn}\label{defharmonic}
We call $u \in \cF^D_{loc}(M)\cap L^{\infty}_{loc}(D)$ satisfying (A) and (B) an $\cE$-harmonic map on $D$ if
\begin{align}\label{harmonicmapeq}
\cE(u,\psi)=0
\end{align}
for all $\psi \in \Pi_u\cC_D(\mathbb{R}^d)$.
\end{dfn}
Viewing \cite{Oka24} and \pref{ELeq}, $\cE$-harmonic maps admit a characterization through stochastic processes. Throughout this article, we assume that the Dirichlet form has no killing part. Therefore, the statement below is a simplified version of the corresponding result in \cite{Oka24}. 
\begin{thm}[cf. \cite{Oka24}]\label{harmonicmartingale}
Let $M$ be a properly embedded Riemannian submanifold of $\mathbb{R}^d$. Let $(\cE,\cF)$ be a regular Dirichlet form on $L^2(E;\bdm)$ without killing term. We assume that a Borel measurable map $u \in \mathcal{F}_{loc}^D(M) \cap L^{\infty}_{loc}(D)$ is quasi-continuous on $D$ and satisfies \thetag{A} and \thetag{B}. Then $u$ is $\cE$-harmonic on $D$ if and only if for each relatively compact open set $D_1$ with $\overline{D_1} \subset D$, the process $u(Z)^{\tau_{D_1}}$ is a $\left(\Prob_z, \{\cF^Z_t\}_{t\geq 0} \right)$-$\gamma$-martingale for q.e. $z\in E$, where the connection rule $\gamma$ is given by \eqref{ConneEmbed}.
\end{thm}
We also consider the following assumption for the Dirichlet form:
\begin{description}
\item[\thetag{C}] The Dirichlet form $(\cE,\cF)$ has no killing part, i.e. $\sk=0$, and $C_0^{\infty}(E)$ is an operator core of the generator $\cL$ on  $L^2(E;\bdm)$ associated with $(\cE,\cF)$. In addition, there exists a smooth vector field $\cH \in \mathfrak{X}(E)$ such that
\item[\thetag{C-1}] for any $u,v \in L^1_{loc}(E; \bdm)$ with $\cH u, \cH v \in L^1_{loc}(E;\bdm)$ and $uv, (\cH u)v, u (\cH v) \in L^1(E;\bdm)$, it holds that
\[
\int_E (\cH u) v \, d\bdm=-\int_E u (\cH v) \, d\bdm;
\]
\item[\thetag{C-2}]for all $\phi \in C_0^{\infty}(E)$, $\cH \cL \phi \in L^2(E; \bdm)$ and
\[
[\cL, \cH]\phi=0.
\]
\end{description}
Our main theorem is as follows.
\begin{thm}\label{DiffHarmonic}
Let $(E,h)$ be a second countable Riemannian manifold and $D$ an open set in $E$. Let $(M,g)$ be a properly embedded Riemannian submanifold of $\mathbb{R}^d$. Let $(\cE,\cF)$ be a regular Dirichlet form satisfying \thetag{C} for a vector field $\cH \in \mathfrak{X}(E)$.
Let $u \in \cF^D_{loc}(M)\cap L^{\infty}_{loc}(D)$ be an $\cE$-harmonic map on $D$ satisfying \thetag{A} and \thetag{B}.
Let $\cH u$ be the distributional derivative of $u$ with respect to $\cH$.
We assume that
$\cH u$ is in $\cF^{D}_{loc}(\mathbb{R}^d) \cap L^{\infty}_{loc}(D)$ and satisfies conditions \thetag{A} and \thetag{B}. We fix a $TM$-valued Borel measurable quasi-continuous modification of $(u,\cH u)$ and set $\bJ=(u(Z), \cH u(Z))$. Then for each relatively compact open set $D_1$ with $\overline{D_1} \subset D$, the process $\bJ^{\tau_{D_1}}$ is a $\left(\Prob_z, \{\cF^Z_t\}_{t\geq 0} \right)$-$\gamma^{\bcl}$-martingale for q.e. $z \in E$. In particular, $\bJ^{\tau_{D_1}}$ is a $TM$-valued semimartingale.
\end{thm}
\begin{rem}
Under the assumption of \tref{DiffHarmonic}, $\cH u$ satisfies $\cH u=\Pi_u \cH u$, $\bdm$-a.e. on $D$. Hence its quasi-continuous modification satisfies $\cH u=\Pi_u \cH u$ q.e. on $D$. Therefore, we can take a $TM$-valued Borel measurable quasi-continuous modification of $(u,\cH u)$.
\end{rem}
We also introduce stochastic parallel transport of tangent vectors along $M$-valued \cl semimartingales that is compatible with the connection rule $\gamma$. Using this parallel transport, we linearize $\gamma^{\bcl}$-martingales and obtain a non-local mean-value property of the differential of harmonic maps in \tref{MeanValue}.

We further obtain the derivative formula for harmonic maps associated with isotropic L\'evy processes on compact Riemannian manifolds having a nonzero Brownian component in \tref{diffLevy}. This differential formula is a non-local counterpart of Theorem 5.3 in \cite{AT98_2}, which gives a derivative formula of harmonic map heat flows involving the Ricci curvature of the domain manifold and the curvature tensor of the target manifold.
Isotropic L\'evy processes on Riemannian manifolds were constructed in \cite{AppleEst00} and the integration-by-parts formulas for those L\'evy processes were studied in \cites{KaiTake21, KaiTake212}. \tref{diffLevy} in this article requires a nonzero Brownian component and therefore does not cover the pure jump case. In addition, we also restrict the differentiation direction to vector fields in \thetag{C} on the domain. Consequently, our formula is not a direct extension of the corresponding differential formulas for scalar-valued functions. Nevertheless, the contribution of jumps can be written explicitly in terms of the second fundamental form on the target manifold.\\

We give an outline of the paper. In \sref{mtgltan}, we recall some basic notions and facts regarding the differential geometry of tangent bundles. In particular, we will focus on the notions of complete lifts of tensors and connections. We refer to \cite{YanoIshi73} for details. We also construct the connection rule $\gamma^{\bcl}$ and the corresponding It\^o calculus for discontinuous semimartingales on $TM$. In \sref{proof}, we prove \tref{DiffHarmonic}. \sref{Examples} provides some examples satisfying the assumption for \tref{DiffHarmonic}. In \sref{SectionParallel}, we introduce the stochastic parallel transport along discontinuous semimartingales on manifolds with projected jumps. Using the parallel transport, we also obtain the linearized local martingale identity for differentials of harmonic maps and a differential formula of harmonic maps for isotropic L\'evy processes on compact Riemannian manifolds.\\

Throughout this article, for $a, b \in \mathbb{R}$, we abbreviate $\max \{a,\, b\}$ and $\min \{a,\, b \}$ as $a\lor b$ and $a \land b$, respectively. For a stochastic process $H$ and a stopping time $\tau$, we write the stopped process as $H^{\tau}$ defined by
\[
H^{\tau}_t(\omega)=H_{t\land \tau (\omega)}(\omega).
\]
For a \cl process $X$, we denote by $X_-$ the process obtained by taking the left-limit of $X$, namely,
\[
X_{t-}(\omega)=\lim_{s\nearrow t}X_s(\omega).
\]
Given a topological space $M$, we denote by $C_0(M)$ the set of all continuous functions on $M$ with compact support. In the case that $M$ is a manifold, we denote by $C_0^{\infty}(M)$ the set of all $C^{\infty}$ functions with compact support.

We use the following notational conventions:
\begin{itemize}
\item $x,y$ denote elements in $M$, which is a target manifold;
\item $\bU, \bV$ denote elements in $TM$;
\item $\cX, \cY$ denote vector fields on $M$;
\item $\wilde{\bU}, \wilde{\bV}$ denote elements in $TTM$;
\item $z, w$ denote elements in $E$, which is a domain manifold.
\end{itemize}

\section{It\^o calculus for discontinuous semimartingales on tangent bundles}\label{mtgltan}
In this section, we introduce a connection rule on tangent bundles and formulate the It\^o calculus. To begin with, we recall some basic notions regarding tangent bundles. Let $M$ be an $n$-dimensional manifold. We denote the canonical projections of $TM$ and $TTM$ by
\[
\pi_{TM} \colon TM \to M,\ \pi_{TTM} \colon TTM \to TM,
\]
respectively. For any smooth curve $\bU(t)$ in $TM$, we can take a smooth map $(s,t)\mapsto x(s,t)\in M$ in such a way that $\bU(t)$ can be represented as
\[
\bU(t)=\left( \frac{\partial}{\partial s} \right)_{s=0} x(s,t).
\]
Thus every tangent vector in $TTM$ can be written as $\dis \left( \frac{\partial}{\partial t} \right)_{t=0}\left( \frac{\partial}{\partial s} \right)_{s=0} x(s,t)$. Then
\begin{align*}
\pi_{TTM}(\bU'(0))&=\bU(0)\\
&=\left( \frac{d}{d s} \right)_{s=0} x(s,0),\\
d\pi_{TM}\left( \bU'(0) \right) &= \left( \frac{\partial}{\partial t} \right)_{t=0} \pi_{TM} \left( \left( \frac{\partial}{\partial s} \right)_{s=0} x(s,t) \right) \\
&= \left( \frac{d}{d t}\right)_{t=0}x(0,t).
\end{align*}
Note that $d\pi_{TM} \colon TTM \to TM$ is also a vector bundle and if we define $s_M \colon TTM \to TTM$ by
\[
s_M \left( \left( \frac{\partial}{\partial t} \right)_{t=0}\left( \frac{\partial}{\partial s} \right)_{s=0} x(s,t) \right) = \left( \frac{\partial}{\partial s} \right)_{s=0}\left( \frac{\partial}{\partial t} \right)_{t=0}x(s,t),
\]
then $s_M$ is a bundle isomorphism between $\pi_{TTM}\colon TTM \to TM$ and $d\pi_{TM}\colon TTM \to TM$ satisfying
\[
\pi_{TTM} \circ s_M=d\pi_{TM},\ d\pi_{TM}\circ s_M = \pi_{TTM},\ s_M^2=\mathrm{id}_{TTM}.
\]

Next we recall the complete and vertical lifts of tensors. For a function $f\in C^{\infty}(M)$, the complete lift $f^{\bcl}\in C^{\infty}(TM)$ is defined by $f^{\bcl}(\bU)=\langle df, \bU \rangle$ for $\bU \in TM$. On the other hand, the vertical lift of a smooth function $f \in C^{\infty}(M)$ is defined by $f^{\bvl}:=\pi_{TM}^*f$. Then for a vector field $\mathcal{X} \in \mathfrak{X}(M)$, the complete lift $\mathcal{X}^{\bcl} \in \mathfrak{X}(TM)$ is uniquely determined as a vector field satisfying
\[
\mathcal{X}^{\bcl}f^{\bcl}=(\mathcal{X}f)^{\bcl},\ \mathcal{X}^{\bcl}f^{\bvl}=(\cX f)^{\bvl}
\]
for any $f\in C^{\infty}(M)$. The vertical lift of a vector field $\cX \in \mathfrak{X}(M)$ is defined through
\[
\cX^{\bvl}f^{\bcl} = (\cX f)^{\bvl},\ \cX^{\bvl}f^{\bvl} = 0
\]
for any $f\in C^{\infty}(M)$. For a 1-form $\alpha \in \Omega^1(M)$, the complete lift $\alpha^{\bcl} \in \Omega^1(TM)$ is defined through
\[
\langle \alpha^{\bcl}, \mathcal{X}^{\bcl} \rangle=\langle \alpha, \mathcal{X} \rangle^{\bcl},\ \langle \alpha^{\bcl}, \mathcal{X}^{\bvl} \rangle=\langle \alpha, \mathcal{X} \rangle^{\bvl},
\]
for all $\mathcal{X} \in \mathfrak{X}(M)$.
The vertical lift of a 1-form $\alpha \in \Omega^1(M)$ is defined through
\[
\langle \alpha^{\bvl}, \cX^{\bvl} \rangle=0,\ \langle \alpha^{\bvl}, \cX^{\bcl} \rangle = \langle \alpha, \cX \rangle^{\bvl}
\]
for all $\cX \in \mathfrak{X}(M)$.
Then the complete lift of 2-tensors can be defined through
\[
(\alpha \otimes \beta)^{\bcl}=\alpha^{\bcl} \otimes \beta^{\bvl}+ \alpha^{\bvl} \otimes \beta^{\bcl}
\]
for $\alpha, \beta \in \Omega^1(M)$.

Hereafter, we suppose that $(M,g)$ is a Riemannian submanifold of the higher-dimensional Euclidean space $\mathbb{R}^d$ again. Continuing from the previous section, we denote the isometric embedding of $M$ by $\iota \colon M \to \mathbb{R}^d$. Then $TM$ and $TTM$ are regarded as submanifolds of the higher-dimensional Euclidean spaces by the embedding $\iota_* \colon TM \to \mathbb{R}^d\times \mathbb{R}^d$ and $\iota_{**} \colon TTM \to \mathbb{R}^d \times \mathbb{R}^d \times \mathbb{R}^d \times \mathbb{R}^d$, respectively. Here the symbol $*$ stands for the differential. Let $\pi_i \colon \mathbb{R}^d \times \mathbb{R}^d \to \mathbb{R}^d$ ($i=1,2$) be the projection to the $i$-th $\mathbb{R}^d$-component. In the same way, for $i=1,2,3,4$, let $\tilde{\pi}_i \colon \mathbb{R}^d \times \mathbb{R}^d \times \mathbb{R}^d \times \mathbb{R}^d \to \mathbb{R}^d$ be the projection onto the $i$-th $\mathbb{R}^d$-component.
We set the vertical subspace in $T_{\bU}TM$ as $V_{\bU}M:= \mathrm{ker}\, d\pi_{\bU}$. Then $V_{\bU}M$ can be naturally identified with $T_{\pi_{TM}\bU}M$. In general, for a given connection $\nabla$ on $M$, we can determine the horizontal subspace $H_{\bU}M$ in $T_{\bU}TM$ for each $\bU \in TM$ as follows:
\[
H_{\bU}:=\{ \bU'(0) \mid \bU(t)\ \text{is a smooth curve on}\ TM\ \text{with}\ \bU(0)=\bU,\ \nabla_{\frac{d}{dt}}\bU(t)=0 \}.
\]
Then every vector $\wilde{\bU} \in T_uTM$ can be uniquely decomposed as $\wilde{\bU}=\wilde{\bU}^H+\wilde{\bU}^V$, where $\wilde{\bU}^H\in H_{\bU}M$, $\wilde{\bU}^V\in V_{\bU}M$. Define the map $K\colon TTM \to TM$ by $K(\wilde{\bU})=\wilde{\bU}^V$. Here we identify $V_{\bU}M$ with $T_{\pi_{TM}(\bU)}M$. Then for every smooth curve $\bU(t)$ on $TM$, it holds that
\[
K(\bU'(0))=\nabla_{\frac{d}{dt}}\bU(0).
\]
Now since $M$ is embedded in $\mathbb{R}^d$, the horizontal and vertical components of $\wilde{\bU} \in T_{\bU}TM$ can be written as
\begin{align*}
\wilde{\bU}^H&=(x'(0),\II(x'(0),\bU)),\\
\wilde{\bU}^V&=K(\bU'(0))\\
&=(0,\nabla_{\frac{d}{dt}}\bU(0))\\
&=(0,\Pi_x\tilde \pi_4\bU'(0))\\
&=(0,\Pi_x\tilde \pi_4\wilde{\bU}).
\end{align*}
Thus the horizontal-vertical decomposition of $\wilde{\bU}$ can be written as
\begin{align}\label{TTMdecom}
\wilde{\bU}=(x'(0),\II(x'(0),\bU))+(0, \Pi_x\tilde \pi_4\wilde{\bU}).
\end{align}
We denote the vertical and horizontal lifts of tangent vectors by $v_{\bU}, h^{\nabla}_{\bU} \colon T_xM \to T_{\bU}TM$ for $\bU \in T_xM$. Then for a vector field $\cX \in \mathfrak{X}(M)$ and $\bU \in TM$, it holds that
\begin{align}\label{clhor}
v_{\bU}(\cX)&=\cX^{\bvl}(\bU),\nonumber\\
\cX^{\bcl}(\bU)&=h^{\nabla}_{\bU}(\cX)+v_{\bU}(\nabla_{\bU}\cX).
\end{align}

Let $g^{\bcl}$ be the complete lift of the Riemannian metric $g$. Then $g^{\bcl}$ is a pseudo-Riemannian metric on $TM$. Moreover, the complete lift $\nabla^{\bcl}$ of the Levi-Civita connection $\nabla$ is also the Levi-Civita connection associated with $g^{\bcl}$. In addition, the differential map $\iota_* \colon TM \to \mathbb{R}^d \times \mathbb{R}^d$ is an isometric embedding with respect to the metric $\delta_{\mathbb{R}^d}^{\bcl}$, which is the complete lift of the Euclidean metric $\delta_{\mathbb{R}^d}$ on $\mathbb{R}^d$, specified by
\[
\langle (\ba_1,\ba_2), (\bb_1,\bb_2) \rangle_{\delta_{\mathbb{R}^d}^{\bcl}}=\langle \ba_1, \bb_2 \rangle_{\delta_{\mathbb{R}^d}} + \langle \ba_2, \bb_1 \rangle_{\delta_{\mathbb{R}^d}},\ \ba_1,\ba_2,\bb_1,\bb_2 \in \mathbb{R}^d.
\]
Indeed, it holds that
\begin{align*}
(\iota_*)^*\delta_{\mathbb{R}^d}^{\bcl}&=(\iota^*\delta_{\mathbb{R}^d})^{\bcl}=g^{\bcl}.
\end{align*}
In order to define the notion of discontinuous martingales on tangent bundles, we need a connection rule on $TM$. \pref{cpltgamma} below provides us with a connection rule on the tangent bundle which is naturally obtained from the connection rule $\gamma$ on $M$ defined by \eqref{ConneEmbed}.

\begin{prop}\label{cpltgamma}
Define the map $\gamma^{\bcl} \colon TM\times TM \to TTM$ by
\[
\gamma^{\bcl} (\bU,\bV):=s_M(d\gamma (\bU,\bV)).
\]
Then for $\bU \in T_xM,\bV \in T_yM$,
\begin{align}\label{conntan}
\gamma^{\bcl} (\bU,\bV)=(\gamma (x, y), \II(\gamma(x,y),\bU)) + (0,-A_{\Pi_x^{\perp}(y-x)}\bU+ \Pi_x(\bV-\bU)))
\end{align}
and $\gamma^{\bcl}$ is a connection rule on $TM$.
\end{prop}
\begin{proof}
Let $x(\ep)$ and $y(\ep)$ be curves with $x(0)=x$, $x'(0)=\bU$, $y(0)=y$, $y'(0)=\bV$. It holds that
\begin{align*}
d\gamma (\bU,\bV)&=\left( \frac{d}{d \ep} \right)_{\ep=0}  \Pi_{x(\ep)}(y(\ep)-x(\ep))\\
&=\left( \frac{d}{d \ep} \right)_{\ep=0}  \Pi_{x(\ep)}(y-x)+\Pi_x(\bV-\bU).
\end{align*}
We check that it holds that
\begin{align*}
\left( \frac{d}{d \ep} \right)_{\ep=0}  \Pi_{x(\ep)}\bW&=\II(\bU,\bW)\ \text{for}\ \bW \in T_xM,\\
\left( \frac{d}{d \ep} \right)_{\ep=0}  \Pi_{x(\ep)}\xi&=-A_{\xi}\bU \ \text{for}\ \xi \in T_x^{\perp}M.
\end{align*}
Let $\{e_i\}_{i=1}^n$ be an orthonormal basis on $T_xM$ and $e_i(\ep)$ the parallel displacement of $e_i$ along $x(\ep)$ with respect to the Levi-Civita connection $\nabla$. Then for $\bW \in T_xM$,
\begin{align*}
\left( \frac{d}{d \ep} \right)_{\ep=0}  \Pi_{x(\ep)}\bW&=\left( \frac{d}{d \ep} \right)_{\ep=0}\sum_{i=1}^n g(e_i(\ep),\bW ) e_i(\ep) \\
&=\sum_{i=1}^n\left(\langle \II(e_i,\bU), \bW \rangle_{\delta_{\mathbb{R}^d}} e_i + g( e_i,\bW ) \II(e_i,\bU)  \right)\\
&=\II(\bU,\bW).
\end{align*}
Here we used $\dis \left(\frac{d}{d\ep}\right)_{\ep=0}e_i(\ep)=\II(e_i,\bU)$. In a similar way, for $\xi \in T_x^{\perp}M$, we have
\begin{align*}
\left( \frac{d}{d \ep} \right)_{\ep=0}  \Pi_{x(\ep)}\xi&=\left( \frac{d}{d \ep} \right)_{\ep=0}\sum_{i=1}^n\langle e_i(\ep),\xi \rangle_{\delta_{\mathbb{R}^d}} e_i(\ep) \\
&=\sum_{i=1}^n\langle \II(e_i,\bU), \xi \rangle_{\delta_{\mathbb{R}^d}} e_i \\
&=-\sum_{i=1}^n g(e_i, A_{\xi}\bU) e_i \\
&=-A_{\xi}\bU.
\end{align*}
Therefore, we obtain \eqref{conntan}. Next we check that $\gamma^{\bcl}$ satisfies the conditions in the definition of connection rules. Obviously $\gamma^{\bcl}$ is a smooth map on $TM \times TM$ satisfying (i) and (ii). We will check the condition (iii). Let $\wilde{\bV} \in T_{\bU}TM$ and $\bU(t)$ a smooth curve on $TM$ with $\bU(0)=\bU$, $\bU'(0)=\wilde{\bV}$, $\pi_{TM}\bU(t)=x(t)$. Then
\begin{align*}
d \gamma^{\bcl} (\bU,\cdot)_{\bU} \wilde{\bV} &= \left( \frac{d}{dt}\right)_{t=0} \gamma^{\bcl} (\bU,\bU(t))\\
&=\left( \frac{d}{dt}\right)_{t=0} (\gamma (\pi_{TM} \bU,\pi_{TM} \bU(t)),\\
&\h \II(\gamma(\pi_{TM} \bU,\pi_{TM} \bU(t)),\bU)-A_{\Pi_x^{\perp}(x(t)-x)}\bU+ \Pi_x(\pi_2 \bU(t)-\pi_2 \bU)))\\
&=(x'(0),\II_x (x'(0),\bU)-A_{\Pi_x^{\perp}x'(0)}\bU+\Pi_x\tilde \pi_4\wilde{\bV})\\
&=\wilde{\bV},
\end{align*}
where we used $\Pi_x^{\perp}x'(0)=0$ and \eqref{TTMdecom} in the last equality. Therefore, we have $d\gamma^{\bcl}(\bU,\cdot)_{\bU}=id_{T_{\bU}TM}$.
\end{proof}

\begin{rem}
The torsion-free connection $\widetilde \nabla$ on $TM$ induced by the connection rule $\gamma^{\bcl}$ is the complete lift of the Levi-Civita connection on $M$. In fact, let $x_0\in M$ and $(x^k)$ a local coordinate with $x^k(x_0)=0$. Let $(x^k,u^k)$ be a coordinate associated with the local trivialization of $TM$. Denote the local expression of $\gamma$ and $\gamma^{\bcl}$ by $G(x^k,y^k)$ and $\widetilde G(x^k,y^k,u^k,v^k)$, respectively. Then
\begin{align*}
\widetilde G (x^k,y^k,u^k,v^k)&=s_M\left( \left(x^k,G^k,u^k,\frac{\partial G^k}{\partial x^l}u^l+\frac{\partial G^k}{\partial y^l}v^l\right) \right)\\
&=\left( x^k,u^k,G^k,\frac{\partial G^k}{\partial x^l}u^l+\frac{\partial G^k}{\partial y^l}v^l \right).
\end{align*}
Let $\Gamma^k_{ij}$ be the Christoffel symbols associated with the Levi-Civita connection. Then the relationship between connections and connection rules is written as
\[
\left.\frac{\partial^2G^k}{\partial y^i \partial y^j}(0,y^l)\right|_{y^l=0}=\Gamma^k_{ij}(0)
\]
since the connection rule $\gamma$ induces the Levi-Civita connection. We denote the Christoffel symbols associated with $\widetilde \nabla$ by $\widetilde \Gamma$ which are determined through
\begin{align*}
\widetilde \nabla_{\frac{\partial}{\partial x^i}}\frac{\partial}{\partial x^j}&=\widetilde{\Gamma}^k_{ij}\frac{\partial}{\partial x^k}+\widetilde{\Gamma}^{\bar{k}}_{ij}\frac{\partial}{\partial u^k},\\
\widetilde \nabla_{\frac{\partial}{\partial x^i}}\frac{\partial}{\partial u^j}&=\widetilde{\Gamma}^k_{i\bar{j}}\frac{\partial}{\partial x^k}+\widetilde{\Gamma}^{\bar{k}}_{i\bar{j}}\frac{\partial}{\partial u^k},\\
\widetilde \nabla_{\frac{\partial}{\partial u^i}}\frac{\partial}{\partial x^j}&=\widetilde{\Gamma}^k_{\bar{i}j}\frac{\partial}{\partial x^k}+\widetilde{\Gamma}^{\bar{k}}_{\bar{i}j}\frac{\partial}{\partial u^k},\\
\widetilde \nabla_{\frac{\partial}{\partial u^i}}\frac{\partial}{\partial u^j}&=\widetilde{\Gamma}^k_{\bar{i}\bar{j}}\frac{\partial}{\partial x^k}+\widetilde{\Gamma}^{\bar{k}}_{\bar{i}\bar{j}}\frac{\partial}{\partial u^k},
\end{align*}
where we use the index $(i,j,k)$ for the components of $\frac{\partial}{\partial x^i}, \frac{\partial}{\partial y^i}$ and $(\bar{i}, \bar{j}, \bar{k})$ for those of $\frac{\partial}{\partial u^i}, \frac{\partial}{\partial v^i}$. Then by a simple calculation, we obtain
\begin{align*}
\widetilde{\Gamma}^k_{ij}(0,u^k)&=\frac{\partial^2 G^k}{\partial y^i \partial y^j}(0,0)=\Gamma^k_{ij}\\
\widetilde{\Gamma}^k_{\bar{i}j}(0,u^k)&=\frac{\partial^2 G^k}{\partial v^i \partial y^j}(0,0)=0\\
\widetilde{\Gamma}^k_{\bar{i}\bar{j}}(0,u^k)&=\frac{\partial^2 G^k}{\partial v^i \partial v^j}(0,0)=0\\
\widetilde{\Gamma}^{\bar{k}}_{\bar{i}j}(0,u^k)&=\frac{\partial^2}{\partial v^i \partial y^j}\left(\frac{\partial G^k}{\partial x^l}u^l+\frac{\partial G^k}{\partial y^l}v^l \right)=\frac{\partial^2 G^k}{\partial y^j \partial y^l}(0,0)\delta_i^l=\Gamma^k_{ij},\\
\widetilde{\Gamma}^{\bar{k}}_{\bar{i}\bar{j}}(0,u^k)&=\frac{\partial^2}{\partial v^i \partial v^j}\left(\frac{\partial G^k}{\partial x^l}u^l+\frac{\partial G^k}{\partial y^l}v^l \right)=0.
\end{align*}
Finally, note that it holds that
\begin{align*}
\frac{\partial \Gamma^k_{ij}}{\partial x^l}(0)&= \left(\frac{\partial}{\partial x^l}\right)_{x=0}\left(\frac{\partial^2 G^k}{\partial y^i \partial y^j}(x^k,x^k)\right)\\
&= \frac{\partial^3 G^k}{\partial x^l \partial y^i \partial y^j}(0,0)+\frac{\partial^3 G^k}{\partial y^l \partial y^i \partial y^j}(0,0)
\end{align*}
Therefore, we obtain
\begin{align*}
\widetilde{\Gamma}^{\bar{k}}_{ij}(0,u^k)&=\left(\frac{\partial^2}{\partial y^i \partial y^j}\right)_{y=0} \left(\frac{\partial G^k}{\partial x^l}(0,y^k)u^l+\frac{\partial G^k}{\partial y^l}(0,y^k)u^l \right)=\frac{\partial \Gamma^k_{ij}}{\partial x^l}(0)u^l.
\end{align*}
Those Christoffel symbols coincide with those of the complete lift of the connection.
\end{rem}
Let $\lambda_g \colon TM \to T^*M$ be the bundle isomorphism induced by the Riemannian metric $g$. We also denote the isomorphism from $TTM$ to $T^*TM$ induced by the pseudo-Riemannian metric $g^{\bcl}$ by $\lambda_{g^{\bcl}}$. 
\begin{prop}\label{inthorvert}
Let $\bJ=(X,J)$ be a $TM$-valued semimartingale and $\wilde{\cV}$ a $TTM$-valued adapted \cl process above $\bJ$.
\begin{itemize}
\item[(1)]If $\wilde{\cV}_-\in V_{\bJ_-}M$, then
\[
\int \lambda_{g^{\bcl}}\left( \wilde{\cV}_- \right) d^{\gamma^{\bcl}}\bJ = \int \langle v_{\bJ_-}^{-1}(\wilde{\cV}_-), dX\rangle_{\delta_{\mathbb{R}^d}}.  
\]
\item[(2)]If $\wilde{\cV}_-\in H_{\bJ_-}M$, then
\begin{align*}
\int \lambda_{g^{\bcl}}\left( \wilde{\cV}_- \right) d^{\gamma^{\bcl}}\bJ &= \int \langle (h^{\nabla}_{\bJ_-})^{-1}(\wilde{\cV}_-), dJ \rangle_{\delta_{\mathbb{R}^d}} + \frac{1}{2}\int \langle \II((h^{\nabla}_{\bJ_-})^{-1} (\wilde{\cV}_-), J_-), \II(dX^c, dX^c)\rangle_{\delta_{\mathbb{R}^d}} \\
&\h +\sum_{0<s\leq \cdot} \langle \II( (h^{\nabla}_{\bJ_{s-}})^{-1}(\wilde{\cV}_{s-}),J_{s-}),  \gamma^{\perp}(X_{s-},X_s) \rangle_{\delta_{\mathbb{R}^d}},
\end{align*}
where we set $\gamma^{\perp}(x,y)=\Pi^{\perp}_x(y-x)$ for $x,y\in M$ and denote the continuous local martingale part of an $\mathbb{R}^d$-valued semimartingale $X$ by $X^c$.
\end{itemize}
\end{prop}
\begin{proof}
First, we assume $\wilde{\cV}_- \in V_{\bJ_-}M$. Then
\begin{align*}
\lambda_{g^{\bcl}}(\wilde{\cV}_-)\, d^{\gamma^{\bcl}}\bJ&=\langle \wilde{\cV}_-, d^{\gamma^{\bcl}}\bJ \rangle_{\delta_{\mathbb{R}^d}^{\bcl}}\\
&=\langle v^{-1}_{\bJ_-}(\wilde{\cV}_-), dX \rangle_{\delta_{\mathbb{R}^d}}.
\end{align*}
Thus we have (1). Next we assume $\wilde{\cV}_- \in H_{\bJ_-}M$. Then we have
\begin{align*}
\lambda_{g^{\bcl}}\left( \wilde{\cV}_- \right) d^{\gamma^{\bcl}}\bJ &=\langle \wilde{\cV}_-,d^{\gamma^{\bcl}}\bJ\rangle_{\delta_{\mathbb{R}^d}^{\bcl}}\\
&=\langle \wilde{\cV}_-, h^{\nabla}_{\bJ_-}\left( d^{\gamma}X \right) \rangle_{\delta_{\mathbb{R}^d}^{\bcl}}+\langle \wilde{\cV}_-, v_{\bJ_-}\left( \Pi_{X_-}dJ \right) \rangle_{\delta_{\mathbb{R}^d}^{\bcl}} - \langle \wilde{\cV}_-, v_{\bJ_-}(A_{\Pi^{\perp}dX}J_-) \rangle_{\delta_{\mathbb{R}^d}^{\bcl}}\\
&=\langle h^{\nabla -1}_{\bJ_-}(\wilde{\cV}_-), dJ \rangle_{\delta_{\mathbb{R}^d}} + \langle  h^{\nabla -1}_{\bJ_-}(\wilde{\cV}_-), -A_{\Pi^{\perp}dX}J_- \rangle_{\delta_{\mathbb{R}^d}}.
\end{align*}
Let $\bi$ be the extension of the embedding $\iota \colon M \to \mathbb{R}^d$ to a tubular neighborhood of $M$, obtained by composing $\iota$ with the normal projection. Then, for $x \in M$,
\[
D \bi (x)=\Pi_x,\ D^2\bi(x)(\bU,\bV)=\II_x(\bU,\bV)
\]
for $\bU,\bV \in T_xM$. Therefore,
\begin{align*}
\langle  (h^{\nabla}_{\bJ_-})^{-1}(\wilde{\cV}_-), -A_{\Pi^{\perp}dX}J_- \rangle_{\delta_{\mathbb{R}^d}} &= \langle \II( (h^{\nabla}_{\bJ_-})^{-1}(\wilde{\cV}_-),J_-), \Pi^{\perp}dX\rangle_{\delta_{\mathbb{R}^d}} \\
&= \langle \II( (h^{\nabla}_{\bJ_-})^{-1}(\wilde{\cV}_-),J_-), d\bi(X)-\langle D\bi(X_-), d\bi(X)\rangle \rangle_{\delta_{\mathbb{R}^d}} \\
&= \frac{1}{2}\langle \II( (h^{\nabla}_{\bJ_-})^{-1}(\wilde{\cV}_-),J_-), \II (dX^c,dX^c) \rangle_{\delta_{\mathbb{R}^d}} \\
&\h + \langle \II( (h^{\nabla}_{\bJ_-})^{-1}(\wilde{\cV}_-),J_-),  \gamma^{\perp}(X_-,X) \rangle_{\delta_{\mathbb{R}^d}}.
\end{align*}
Thus we have (2).
\end{proof}

\begin{cor}\label{martveccondition}
Let $\bJ=(X,J)$ be a $TM$-valued semimartingale. Then the following are equivalent:
\begin{itemize}
\item[(i)]$\bJ$ is a $\gamma^{\bcl}$-martingale;
\item[(ii)]$X$ is a $\gamma$-martingale and for any $\cX \in \mathfrak{X}(M)$,
\begin{align*}
&\int \langle \cX(X_-), dJ \rangle_{\delta_{\mathbb{R}^d}} + \frac{1}{2}\int \langle \II(\cX(X_-),\bJ_-),\II(dX^c,dX^c)\rangle_{\delta_{\mathbb{R}^d}}\\
&+\sum_{0<s\leq \cdot} \langle \II( \cX(X_-) ,\bJ_{s-}), \gamma^{\perp}(X_-,X) \rangle_{\delta_{\mathbb{R}^d}}
\end{align*}
is a local martingale.
\item[(iii)]$X$ is a $\gamma$-martingale and for any $\cX \in \mathfrak{X}(M)$, $\int \lambda_{g^{\bcl}}(\cX^{\bcl}(\bJ_-))\, d^{\gamma^{\bcl}}\bJ_s$ is a local martingale.
\end{itemize}
\end{cor}
\begin{proof}
If $\bJ$ is a $\gamma^{\bcl}$-martingale, then for any $\cX \in \mathfrak{X}(M)$, both $\int \langle v_{\bJ_-}(\cX(X_-)), d^{\gamma^{\bcl}}\bJ \rangle_{\delta_{\mathbb{R}^d}^{\bcl}}$ and $\int \langle h_{\bJ_-}^{\nabla}(\cX(X_-)), d^{\gamma^{\bcl}}\bJ \rangle_{\delta_{\mathbb{R}^d}^c}$ are local martingales. Thus by \pref{inthorvert}, $\bJ$ satisfies (ii). Conversely, assume that $\bJ$ satisfies (ii). Let $\wilde{\cV}$ be a $TTM$-valued \cl adapted process above $\bJ$. Then its vertical and horizontal parts $\wilde{\cV}^V$, $\wilde{\cV}^H$ are written as
\begin{align*}
\wilde{\cV}_t^V&=C^i_tv_{\bJ_t}(\cX_i(X_t)),\\
\wilde{\cV}_t^H&=D^i_th^{\nabla}_{\bJ_t}(\cY_i(X_t))
\end{align*}
for some $\mathbb{R}$-valued \cl adapted processes $C^i$ and $D^i$ and vector fields $\cX_i, \cY_i \in \mathfrak{X}(M)$. Therefore, we have
\begin{align*}
\int \lambda_{g^{\bcl}}(\wilde{\cV}_-)\, d^{\gamma^{\bcl}}\bJ &= \int C^i_- d\left(\int \lambda_{g^{\bcl}}(v_{\bJ_-}(\cX_i(X_-)))\, d^{\gamma^{\bcl}}\bJ \right)\\
&\h + \int D^i_- d\left( \int \lambda_{g^{\bcl}}(h_{\bJ_-}(\cY_i(X_-)))\, d^{\gamma^{\bcl}}\bJ \right)
\end{align*}
and the right-hand side is a local martingale by the assumption and \pref{inthorvert}. Finally, the equivalence of (ii) and (iii) is obvious by the relation \eqref{clhor}.
\end{proof}

\section{Proof of \tref{DiffHarmonic}}\label{proof}
Let $(E,h)$ be a second-countable Riemannian manifold and $\bdm$ the associated Riemannian volume measure. Let $(\cE,\cF)$ be a regular Dirichlet form on $L^2(E,\bdm)$. To begin with, we recall several consequences of conditions (A) and (B).
Hereafter, we fix a L\'evy system $(N,H)$ of $Z$, that is, $N$ is a kernel from $E$ to $E_{\varDelta}$, $H$ is a positive continuous additive functional (PCAF) of $Z$ satisfying
\begin{align}\label{LevySys}
\Ex_z \left[ \sum_{0<s\leq t}Y_sF(Z_{s-},Z_s) \right]=\Ex_z \left[ \int_0^tY_s \int_{E_{\varDelta}}F(Z_{s},w)\, N(Z_s,dw)dH_s \right]
\end{align}
for every nonnegative predictable process $Y$ and nonnegative Borel function $F$ on $E_{\varDelta}\times E_{\varDelta}$ satisfying $F(w,w)=0$. Let $\mu_H$ be the Revuz measure of $H$. Then for the jump measure $\sJ$ and killing measure $\sk$ used in \eqref{BDdecom}, we have
\begin{align*}
\sJ(dzdw)&=N(z,dw)\mu_H(dz),\\
\sk(dz)&=N(z,\{\varDelta\})\mu_H(dz).
\end{align*}
\pref{exttest} below is a modification of Lemma 6.7.8 of \cite{ChenFuku} and the proof is completely the same. 
\begin{prop}[cf. \cites{Chen09, ChenFuku}]\label{exttest}
Let $u \in \cF_{loc}^D\cap L^{\infty}_{loc}(D)$ be a Borel measurable function satisfying \thetag{A}. Then for each relatively compact open subset $D_1$ with $\overline{D_1}\subset D$ and $\psi \in \cF^{D_1}\cap L^{\infty}(E)$,
\begin{align*}
\frac{1}{2}\mu_{\langle u,\psi \rangle}^c(D)+\frac{1}{2}\int_{E \times E \backslash \mathrm{diag}(E)}(u(z)-u(w))(\psi(z)-\psi(w))\, \sJ(dzdw)+\int_D u\psi \, d\sk
\end{align*}
is well-defined and finite. We continue to denote the value by $\cE(u,\psi)$. For a relatively compact open set $D_2$ with $\overline{D_1} \subset D_2 \subset \overline{D_2} \subset D$ and $\phi \in C_0(D)\cap \cF$ with $\phi=1$ on $\overline{D_2}$, it also holds that
\[
\cE(u,\psi)=\cE(\phi u, \psi)-\int_{D_1 \times (E \backslash D_2)}\psi(z) \left((1-\phi)u \right)(w)\, \sJ(dzdw).
\]
\end{prop}

\begin{lem}{(\cites{Chen09, ChenFuku}, Lemma 3.5 of \cite{Oka24})}\label{hAB}
Let $D$ be an open set. Let $D_1,D_2$ be relatively compact open sets in $E$ satisfying \eqref{opensets} and $u\in \mathcal{F}^D_{loc}\cap L^{\infty}_{loc}(D)$ a Borel measurable function satisfying \thetag{A} and \thetag{B}. Let $\phi= \phi_{D_2} \in \cF \cap C_0(D)$ be a function satisfying the conditions in \thetag{B}. Then the process
\[
A_t:= \mathbf{1}_{\{t\geq \tau_{D_1}\}}\{(1-\phi (Z_{\tau_{D_1}}))u(Z_{\tau_{D_1}})-(1-\phi (Z_{\tau_{D_1}-}))u(Z_{\tau_{D_1}-})\}
\]
is a process of $\mathbb{P}_z$-integrable variation for q.e. $z\in E$ and the dual predictable projection of $A_t$ is
\[
B_t=\int_0^{t\land \tau_{D_1}}\int_{E\, \backslash \, D_2}(1-\phi)u(z)\, N(Z_s,dz)dH_s.
\]
We further set
\begin{align*}
h_1(z)&:=\mathbb{E}_z\left[ (\phi u)(Z_{\tau_{D_1}}) \right],\\
h_2(z)&:=\mathbb{E}_z \left[((1-\phi )u)(Z_{\tau_{D_1}}) \right].
\end{align*}
Then it holds that $\phi u \in \mathcal{F}^D,\ h_1 \in \mathcal{F}_e,\ \phi u-h_1\in \mathcal{F}_e^{D_1}$, and $\mathbf{1}_{D_1} h_2\in \mathcal{F}_e^{D_1}$. Moreover $h_2=\mathbf{1}_{D_1}h_2+(1-\phi)u$ satisfies \thetag{A} and it holds that for any $\psi \in \cF^{D_1}_e$,
\begin{align}
\mathcal{E}(\mathbf{1}_{D_1}h_2,\psi)&=\int_{D_1 \times E\backslash D_2}\psi(z)\left((1-\phi)u\right)(w)\, \sJ(dzdw)\label{Eh21}
\end{align}
and for any $\psi \in \cF^{D_1}\cap L^{\infty}(D)$,
\begin{align}
\mathcal{E}((1-\phi)u,\psi)&=-\int_{D_1 \times E\backslash D_2}\psi(z)\left((1-\phi)u\right)(w)\, \sJ(dzdw).\label{Eh22}
\end{align}
\end{lem}
\begin{rem}\label{testfunc}
Test functions $\psi$ in \eqref{Eh22} can be taken from $\cF^{D_1} \cap L^{\infty}(E)$ due to \pref{exttest}.
\end{rem}
Under the assumptions of \lref{hAB}, we set $h:=h_1+h_2$, $v:=u-h$. Then by \lref{hAB} and \rref{testfunc}, we have the decomposition
\begin{align}\label{HarmDecom}
u=v + h,
\end{align}
where $v\in \cF_e^{D_1}$ and $h$ satisfies (A) and
\begin{align}\label{heq}
\cE(h, \psi)=0\ \text{for all}\ \psi \in \cF^{D_1}\cap L^{\infty}(E).
\end{align}
We recall the notation for the Fukushima decomposition. For $u\in \cF_e$, there exists a martingale additive functional (MAF) of finite energy $M^{[u]}$ and a continuous additive functional (CAF) of zero energy $N^{[u]}$ such that
\[
\tilde u(Z_t)-\tilde u(Z_0)=M^{[u]}_t + N^{[u]}_t,\ t\geq 0,\ \Prob_z \text{-a.s. for q.e.}\ z\in E
\]
by Theorem 5.2.2 of \cite{FOT}. For a vector valued function $u$, these notations are understood componentwise. We denote the Revuz measure associated with the PCAF $\langle M^{[u]},M^{[u]} \rangle$ by $\mu_{\langle u \rangle}$. Let $D_1$ be a relatively compact open set with $\overline{D_1}\subset D$. For $u \in \cF_e$ and $v\in \cF_e \cap L^2(E;\mu_{\langle u \rangle})$, the stochastic integral $\dis \int v(Z)\, dN^{[u]}$ was defined in \cite{Nakao} and its stopped process at $\tau_{D_1}$ is characterized as a CAF of zero energy satisfying
\[
\lim_{t \to 0} \frac{1}{t}\Ex_{l\bdm}\left[ \int_0^{t\land \tau_{D_1}} v(Z_s)\, dN^{[u]}_s \right]=-\cE(lv, u)
\]
for each $l \in \cF^{D_1}\cap L^{\infty}(E)$ due to Lemma 5.4.4 of \cite{FOT} and Theorem 2.2 of \cite{Nakao}.
\begin{prop}\label{ELeq}
Let $D$ be an open set. Let $D_1,D_2$ be relatively compact open sets in $E$ satisfying \eqref{opensets}.
Let $u\in \cF^D_{loc}(M)\cap L^{\infty}_{loc}(D)$ be an $\cE$-harmonic map. Let $u=v+h$ be the decomposition given by \eqref{HarmDecom}. Then for any $\psi \in \cF^{D_1}(u^*TM)\cap L^{\infty}(E)$, $\cE(v,\psi)=0$.
\end{prop}
\begin{proof}
Let $\{\ep_i \}_{i=1}^d$ be the canonical basis on $\mathbb{R}^d$ and set $P_i(z):=\Pi_{u(z)}\ep_i$. Let $l \in \cF^{D_1}\cap L^{\infty}(E)$. Then there exists a uniformly bounded sequence $l_k \in \cC_{D_1}$ such that
\[
l_k \to l\ \text{in}\ \cE_1,\ \tilde{l}_k \to \tilde{l}\ \text{q.e.}
\]
by Theorem 4.4.3 of \cite{FOT}.
For each $k$, $l_k P_i=\Pi_{u}l_k\ep_i \in \Pi_u\cC_{D_1}$. Hence by \eqref{harmonicmapeq} and \eqref{heq}, we have
\[
\cE(v,l_kP_i)=0.
\]
Therefore, by letting $k \to \infty$, we have
\begin{align}\label{EvlP}
\cE(v,lP_i)=0
\end{align}
since $l_kP_i \to l P_i$ in $\cE_1$ due to the Beurling-Deny decomposition.
Let $\psi \in \cF^{D_1}(u^*TM) \cap L^{\infty}(E)$. Then
\[
\psi = \Pi_u \psi=\sum_{i=1}^d\psi_iP_i.
\]
Since $\psi_i \in \cF^{D_1}\cap L^{\infty}(E)$ for each $i$, we have
\[
\cE(v,\psi)=\sum_{i=1}\cE(v,\psi_iP_i)=0
\]
by \eqref{EvlP}.
\end{proof}
We denote the generator associated with $(\cE,\cF)$ on $L^2(E;\bdm)$ by $\cL$ with its domain $\cD(\cL)$. In general, $\cL$ is a non-positive self-adjoint operator on $L^2(E;\bdm)$.
\begin{lem}\label{DomGen}
Let $D \subset E$ be an open set. Let $u \in \cF_{loc}^D \cap L^{\infty}_{loc}(D)$ be a Borel measurable function satisfying (A). Then for each $v \in C_0(D)\cap \cD(\cL)$, it holds that
\[
\cE(u,v)=\langle u, v_0-v \rangle,
\]
where $v_0=(I-\cL)v$.
\end{lem}
\begin{proof}
Let $G \subset D$ be a relatively compact open subset with $\mathrm{supp}[v] \subset G \subset \overline{G} \subset D$.
Let $\phi \in C_0(D) \cap \cF$ satisfying
\[
0 \leq \phi \leq 1,\ \phi=1\ \text{on}\ G.
\]
Then under the assumption for $v$, we have
\begin{align}
\cE(u,v)&=\frac{1}{2}\mu_{\langle u,v \rangle}^c(D)+\frac{1}{2}\int_{E \times E}(u(z)-u(w))(v(z)-v(w))\, \sJ(dzdw) + \int_D u(z) v(z)\, \sk(dz) \nonumber\\
&=\frac{1}{2}\mu_{\langle \phi u,v \rangle}^c(D)+\frac{1}{2}\int_{E \times E}(\phi u(z)-\phi u(w))(v(z)-v(w))\, \sJ(dzdw)+\int_D \phi u(z) v(z)\, \sk(dz) \nonumber\\
&\h - \int_{E \backslash G} (1-\phi)u(z) \int_G v(w)\, \sJ(dzdw) \nonumber\\
&= \cE(\phi u,v)-\int_E (1-\phi)u(z) \int_G v(w)\, \sJ(dzdw).\label{decEuv}
\end{align}
Since both $\phi u$ and $v$ are in $\cF^D$, it holds that
\begin{align}\label{phiuv}
\cE(\phi u,v)=\langle \phi u, v_0-v \rangle.
\end{align}
On the other hand, for each $l \in \cF$ with $l=0$ on $G$,
\begin{align*}
-\int_E l(z) \int_G v(w)\, \sJ(dzdw)&= \cE(l,v)\\
&=\langle l, v_0-v \rangle.
\end{align*}
Therefore, $-\int_G v(w)\, J(dzdw)=(v_0(z)-v(z))\, \bdm(dz)$ as Radon measures on $E \backslash G$. Since $(1-\phi)u$ is integrable with respect to those measures due to \thetag{A}, we have
\begin{align}\label{(1-phi)uv}
- \int_E (1-\phi)u(z) \int_G v(w)\, J(dzdw) =\int_E (1-\phi)u (v_0-v)\, d\bdm.
\end{align}
Therefore, by \eqref{decEuv}, \eqref{phiuv} and \eqref{(1-phi)uv}, we complete the proof.
\end{proof}
The following Lemma concerns some properties of Dirichlet forms satisfying condition \thetag{C}.
\begin{lem}\label{E1dense}
Let $(\cE,\cF)$ be a regular Dirichlet form on $L^2(E;\bdm)$ satisfying \thetag{C} for some vector field $\cH \in \mathfrak{X}(E)$. Let $\Phi_s$ be the local flow generated by $\cH$. Let $D$ be an open set and $D_1$ a relatively compact open set with $\overline{D_1}\subset D$. Fix $\delta>0$ and a compact set $K$ with $\overline{D_1} \subset K \subset D$ in such a way that
\[
\Phi_s(\overline{D_1})\subset K \ \text{for all}\ s\in [-\delta, \delta].
\]
Then the following hold.
\begin{enumerate}
\item For $s\in [-\delta, \delta]$, $\Phi_s^* \colon C_0^{\infty} (D_1) \to C_0^{\infty}(D)$ extends to an $\cE_1$-isometry from $\cF^{D_1}$ to $\cF^D$ and $s \mapsto \Phi_s^*\phi$ is $\cE_1$-continuous for $\phi \in \cF^{D_1}$ on $[-\frac{\delta}{2},\frac{\delta}{2}]$.
\item For $\phi \in \cF^{D_1}$ with $\cH \phi \in \cF^{D_1}$, there exists a sequence $\phi_k \in C_0^{\infty}(D)$ such that $\phi_k \to \phi$ and $\cH \phi_k \to \cH \phi$ as $k \to \infty$ in $\cE_1$.
\end{enumerate}
\end{lem}
\begin{proof}
\begin{enumerate}
\item
We set
\[
C_K^{\infty}(D):=\{ \phi \in C_0^{\infty}(D)\mid \mathrm{supp}[\phi] \subset K\}.
\]
We show that $\cL \colon C^{\infty}_K(D)\to L^2(E;\bdm)$ is continuous, where we regard $C^{\infty}_K(D)$ as a Fr\'echet space. Let $\{\phi_k\}_{k \in \mathbb{N}}$ be a sequence in $C_K^{\infty}(D)$ and assume that $\phi_k \to \phi$ in $C_K^{\infty}(D)$ and $\cL \phi_k \to f$ in $L^2(E;\bdm)$. Then $\phi_k \to \phi$ in $L^2(E;\bdm)$. Since $\cL$ is closed as an operator on $L^2(E;\bdm)$, we have $\cL \phi=f$. Therefore, $\cL \colon C^{\infty}_K(D)\to L^2(E;\bdm)$ is closed. By the closed graph theorem for Fr\'echet spaces, $\cL$ is continuous.
We take $\phi,\psi \in C_0^{\infty}(D_1)$. Then by the argument above,
\[
\frac{d}{ds}\cL \Phi_s^*\phi=\cL \cH \Phi_s^*\phi\ \text{in}\ L^2(E;\bdm).
\]
Therefore, it holds that
\begin{align*}
\frac{d}{ds}\cE(\Phi_s^*\phi, \Phi_s^*\psi)&=-\frac{d}{ds} \langle \Phi_s^* \phi, \cL \Phi_s^* \psi \rangle_{L^2(E;\bdm)}\\
&=- \langle \cH \Phi_s^*\phi, \cL \Phi_s^*\psi \rangle_{L^2(E;\bdm)}- \langle \Phi_s^*\phi, \cL \cH \Phi_s^*\psi \rangle_{L^2(E;\bdm)}\\
&=0.
\end{align*}
Here in the last equality, we used \thetag{C-1} and \thetag{C-2}. We also have
\[
\frac{d}{ds}\langle \Phi_s^*\phi, \Phi_s^*\psi \rangle_{L^2(E;\bdm)}=\langle \cH \Phi_s^*\phi, \Phi_s^*\psi \rangle_{L^2(E;\bdm)}+\langle \Phi_s^*\phi, \cH \Phi_s^*\psi \rangle_{L^2(E;\bdm)}=0.
\]
Therefore, $s \to \cE_1(\Phi_s^*\phi, \Phi_s^*\psi)$ is constant. In particular, it holds that $\cE_1(\Phi_s^*\phi, \Phi_s^*\psi)=\cE_1(\phi,\psi)$ for $s$ with $|s|\leq \delta$. By condition \thetag{C}, $C_0^{\infty}(E)$ is a special standard core of $(\cE,\cF)$. Thus, by Theorem 4.4.3 of \cite{FOT}, $C_0^{\infty}(D_1)$ is a core of the part form of $(\cE,\cF)$ on $D_1$, which means that $C_0^{\infty}(D_1)$ is $\cE_1$-dense in $\cF^{D_1}$. Therefore, $\Phi_s^*$ can be extended to an isometry from $\cF^{D_1}$ to $\cF^D$.
Moreover, for $\phi \in C_0^{\infty}(D_1)$ and sufficiently small $s,s_0 \in [-\frac{\delta}{2},\frac{\delta}{2}]$,
\begin{align*}
\| \Phi_s^*\phi-\Phi_{s_0}^*\phi \|^2_{\cE_1}&=\cE_1(\Phi_s^*\phi, \Phi_s^*\phi)-2 \cE_1(\Phi_s^*\phi, \Phi_{s_0}^*\phi)+\cE_1(\Phi_{s_0}^*\phi,\Phi_{s_0}^*\phi)\\
&=2(\cE_1(\phi,\phi)-\cE_1(\Phi_{s-s_0}^*\phi, \phi))\\
&=2\langle \phi- \Phi_{s-s_0}^*\phi, \phi - \cL \phi \rangle_{L^2(E;\bdm)}\\
&\to 0\ \text{as}\ s \to s_0,
\end{align*}
where we used the $L^2$-continuity of $s \mapsto \Phi_s^*\phi$ in the third equality. Therefore, $s \mapsto \Phi_s^*\phi$ is $\cE_1$-continuous. Next, for $\phi \in \cF^{D_1}$, we take a sequence $\phi_k \in C_0^{\infty}(D_1)$ with $\phi_k \to \phi$ in $\cE_1$. Then for sufficiently small $s,s_0 \in [-\frac{\delta}{2},\frac{\delta}{2}]$,
\begin{align*}
\|\Phi_s^*\phi-\Phi_{s_0}^*\phi\|_{\cE_1} &\leq \|\Phi_s^*\phi-\Phi_s^*\phi_k\|_{\cE_1} + \|\Phi_s^*\phi_k-\Phi_{s_0}^*\phi_k\|_{\cE_1} + \|\Phi_{s_0}^*\phi_k-\Phi_{s_0}^*\phi\|_{\cE_1}\\
&= 2\|\phi-\phi_k\|_{\cE_1} + \|\Phi_s^*\phi_k-\Phi_{s_0}^*\phi_k\|_{\cE_1}.
\end{align*}
By the argument above, letting $s\to s_0$, we have
\[
\lim_{s \to s_0}\|\Phi_s^*\phi-\Phi_{s_0}^*\phi\|_{\cE_1} \leq 2\|\phi-\phi_k\|_{\cE_1}
\]
for each $k$. Then letting $k \to \infty$, we have
\[
\lim_{s \to s_0}\|\Phi_s^*\phi-\Phi_{s_0}^*\phi\|_{\cE_1}=0.
\]
Therefore, $s\mapsto \Phi_s^*\phi$ is $\cE_1$-continuous for each $\phi \in \cF^{D_1}$.
\item Let $\rho \in C^{\infty}_0(\mathbb{R})$ be a mollifier with
\[
\rho \geq 0,\ \mathrm{supp}[\rho]\subset (-1,1),\ \int_{\mathbb{R}} \rho (t)\, dt =1
\]
and set $\dis \rho_{\ep}(t):= \frac{1}{\ep}\rho \left(\frac{t}{\ep}\right)$.
We fix $\phi \in \cF^{D_1}$ with $\cH \phi \in \cF^{D_1}$. Since $s \to \rho_{\ep}(s)\Phi_s^*\phi$ is $\cE_1$-continuous on $[-\frac{\delta}{2},\frac{\delta}{2}]$ by (1), we can define $\phi_{\ep} \in \cF^{D}$ for $\ep \in (0, \delta)$ by the Bochner integral
\begin{align}\label{Bochner}
\phi_{\ep}:=\int_{\mathbb{R}} \rho_{\ep}(s)\Phi_s^*\phi \, ds
\end{align}
in $\cF^{D}$. Then by (1),
\begin{align*}
\| \phi_{\ep} - \phi \|_{\cE_1} &\leq \int_{\mathbb{R}} \rho_{\ep}(s) \| \Phi_s^*\phi - \phi \|_{\cE_1}\, ds\\
&\to 0\ \text{as}\ \ep \to 0.
\end{align*}
Next, we compute the distributional derivative $\cH \phi_{\ep}$. For each $\eta \in C_0^{\infty}(D)$, $\langle \Phi_s^*\phi_{\ep}, \eta \rangle_{L^2(E;\bdm)}$ is differentiable at $s=0$ and
\begin{align}\label{diffphiep}
\left(\frac{d}{ds}\right)_{s=0} \langle \Phi_s^*\phi_{\ep}, \eta \rangle_{L^2(E;\bdm)}&= \left(\frac{d}{ds}\right)_{s=0}\langle \phi_{\ep}, \Phi_{-s}^*\eta \rangle_{L^2(E;\bdm)} \nonumber \\
&=-\langle \phi_{\ep}, \cH \eta \rangle_{L^2(E;\bdm)}\nonumber \\
&=-\int_{\mathbb{R}}\rho_{\ep}(s)\langle \Phi_s^*\phi, \cH \eta \rangle_{L^2(E;\bdm)}\, ds \nonumber \\
&=-\int_{\mathbb{R}} \rho_{\ep}(s)\langle \phi, \Phi_{-s}^* \cH \eta \rangle_{L^2(E;\bdm)}\, ds \nonumber \\
&=-\int_{\mathbb{R}} \rho_{\ep}(s)\langle \phi, \cH \Phi_{-s}^* \eta \rangle_{L^2(E;\bdm)}\, ds\\
&=\int_{\mathbb{R}} \rho_{\ep}(s)\langle \Phi_s^*\cH \phi, \eta \rangle_{L^2(E;\bdm)}\, ds.\nonumber
\end{align}
Therefore, we have
\[
\cH \phi_{\ep}=\int_{\mathbb{R}}\rho_{\ep}(s)\Phi_s^* \cH \phi \, ds.
\]
In particular, $\cH \phi_{\ep}\in \cF^{D}$ and
\begin{align*}
\| \cH \phi_{\ep}- \cH \phi \|_{\cE_1} &\leq \int_{\mathbb{R}} \rho_{\ep}(s) \| \Phi_s^* \cH \phi-\cH \phi \|_{\cE_1}\, ds\\
&\to 0\ \text{as}\ \ep \to 0
\end{align*}
by (1). On the other hand, from \eqref{diffphiep}, we also have
\begin{align*}
\langle \cH \phi_{\ep}, \eta \rangle_{L^2(E;\bdm)}&=\left\langle \phi, \int_{\mathbb{R}} \rho_{\ep}(s) \frac{d}{ds}\Phi_{-s}^* \eta \, ds \right\rangle_{L^2(E;\bdm)}\\
&=-\left\langle \phi, \int_{\mathbb{R}} \rho'_{\ep}(s) \Phi_{-s}^* \eta \, ds \right\rangle_{L^2(E;\bdm)}\\
&=-\left\langle \int_{\mathbb{R}}\rho'_{\ep}(s)\Phi_s^* \phi \, ds, \eta \right\rangle_{L^2(E;\bdm)}.
\end{align*}
Therefore, we have another representation
\[
\cH \phi_{\ep}=-\int_{\mathbb{R}}\rho'_{\ep}(s)\Phi_s^* \phi \, ds.
\]
Take a sequence $\{\psi_k\}_{k\in \mathbb{N}} \subset C_0^{\infty}(D_1)$ such that $\psi_k \to \phi$ in $\cE_1$. Then
\begin{align*}
\| \psi_{k,\ep}-\phi_{\ep} \|_{\cE_1}&\leq \int_{\mathbb{R}}\rho_{\ep}(t)\| \Phi_t^*(\psi_k-\phi) \|_{\cE_1} \, dt\\
&=\| \psi_k-\phi \|_{\cE_1},
\end{align*}
where $\psi_{k,\ep}$ is defined in the same way as \eqref{Bochner} for $\psi_k$ and we used the consequence of (1). In the same way, we also have
\begin{align*}
\| \cH \psi_{k,\ep}- \cH \phi_{\ep} \|_{\cE_1} &\leq \int_{\mathbb{R}}|\rho'_{\ep}(t)|\| \Phi_t^*(\psi_k-\phi) \|_{\cE_1}\, dt\\
&=\| \rho'_{\ep} \|_{L^1(\mathbb{R})} \| \psi_k-\phi \|_{\cE_1},
\end{align*}
For each $\theta >0$, we choose $\ep>0$ so small that
\[
\|\phi_{\ep}-\phi \|_{\cE_1}+\| \cH \phi_{\ep} - \cH \phi \|_{\cE_1} < \frac{\theta}{2}.
\]
For such $\ep$, we take sufficiently large $k \in \mathbb{N}$ such that
\[
\| \psi_{k,\ep}-\phi_{\ep} \|_{\cE_1}+\| \cH \psi_{k,\ep}-\cH \phi_{\ep} \|_{\cE_1} < \frac{\theta}{2}.
\]
Then $\psi_{k,\ep} \in C_0^{\infty}(D)$ satisfies
\[
\|\psi_{k,\ep}-\phi \|_{\cE_1}+\| \cH \psi_{k,\ep} - \cH \phi\|_{\cE_1} < \theta.
\]
This completes the proof of (2).
\end{enumerate}
\end{proof}

\begin{lem}\label{HInvariance}
Let $D$ be an open set. Let $D_1,D_2$ be relatively compact open sets in $E$ satisfying \eqref{opensets}. Let $u \in \cF^{D}_{loc}\cap L^{\infty}_{loc}(D)$ satisfy \thetag{A} and \thetag{B}. We assume that the Dirichlet form $(\cE,\cF)$ satisfies \thetag{C} with a vector field $\cH \in \mathfrak{X}(E)$.
We further suppose that $\cH u$ is in $\cF^{D}_{loc}\cap L^{\infty}_{loc}(D)$ and satisfies \thetag{A} and \thetag{B}. We set $\bu=(u,\cH u)$. Let $\bu=\bv+\bh$ be the decomposition given by \eqref{HarmDecom} and set $v=\pi_1(\bv)$, $\bv_2=\pi_2(\bv)$. Then for each relatively compact open set $D_0$ with $\overline{D_0} \subset D_1$ and for all $\phi \in \cF^{D_0}$ with $\cH \phi \in \cF^{D_0}$,
\[
\cE(\bv_2,\phi)=-\cE(v,\cH \phi).
\]
\end{lem}
\begin{proof}
Let $\phi \in C_0^{\infty}(D_1)$.
Then by \lref{DomGen},
\begin{align*}
\cE(\cH u, \phi)&= -\langle \cH u, \cL \phi \rangle \\
&= \langle u, \cH \cL \phi \rangle \\
&= \langle u, \cL \cH \phi \rangle \\
&= -\cE(u, \cH \phi),
\end{align*}
where all the pairings above are finite by \lref{DomGen}. Thus by \eqref{heq}, we obtain
\[
\cE(\bv_2, \phi)=-\cE(v,\cH \phi).
\]
Next we take $\phi \in \cF^{D_0}$ with $\cH \phi \in \cF^{D_0}$. By \lref{E1dense}, we can take a sequence $\phi_k \in C_0^{\infty}(D_1)$ such that $\phi_k \to \phi$ and $\cH \phi_k \to \cH \phi$ in $\cE_1$. For each $k \in \mathbb{N}$, we have
\[
\cE(\bv_2, \phi_k)=-\cE (v, \cH \phi_k)
\]
by the argument above. Since $v, \bv_2 \in \cF^{D_1}_e$, by letting $k \to \infty$, we get the desired equality.
\end{proof}

\begin{lem}\label{Nbv}
Let $D$ be an open set. Let $D_1,D_2$ be relatively compact open sets in $E$ satisfying \eqref{opensets}. Let $u \in \cF^D_{loc}(M) \cap L^{\infty}_{loc}(D)$ be an $\cE$-harmonic map on $D$.
Assume that the Dirichlet form $(\cE,\cF)$ satisfies \thetag{C} with a vector field  $\cH$ on $E$.
In addition, assume that $\cH u \in \cF^{D}_{loc}\cap L^{\infty}_{loc}(D)$ and satisfies \thetag{A} and \thetag{B}. Set $\bu:=(u,\cH u)$.
Let
\[
\bu = \bv + \bh
\]
be the componentwise decomposition defined by \eqref{HarmDecom}. Let $v=\pi_1\circ \bv, \bv_2=\pi_2 \circ \bv$. 
Then it holds that for any $\mathcal{X}\in \mathfrak{X}(M)$,
\begin{align*} 
\int_0^{t\land \tau_{D_1}} \langle \mathcal{X}(u(Z_-)), dN^{[\bv_2]}\rangle_{\delta_{\mathbb{R}^d}} = -\int_0^{t\land \tau_{D_1}} \langle \II(\mathcal{X}\circ u, \cH u), dN^{[v]}\rangle_{\delta_{\mathbb{R}^d}}
\end{align*}
for $t\geq 0,\ \Prob_z\text{-a.s. for q.e.}\ z\in E$.
\end{lem}
\begin{proof}
For any $\cX \in \mathfrak{X}(M)$ and $l \in C_0^2(D_1)$, it holds that
\[
\cH (l \cX \circ u)(z) = (\cH l)(z) \cX \circ u(z)+ l(z) \nabla_{\cH u} \cX (u(z))+ l(z) \II (\cH u, \cX\circ u(z))
\]
and the first two terms on the right-hand side are in $\cF^{D_1}(u^*TM)$. Therefore, by \pref{ELeq} and \lref{HInvariance}, we have
\begin{align*}
- \mathcal{E}(l \mathcal{X} \circ u, \bv_2)&= \mathcal{E}(\cH \left(l \mathcal{X}\circ u\right), v)\\
&= \mathcal{E}(l \II(\mathcal{X}\circ u, \cH u),v).
\end{align*}
Next, let $l \in \cF^{D_1}\cap L^{\infty}(E)$. Since $C_0^{\infty}(E)$ is a special standard core of $(\cE,\cF)$ by \thetag{C}, we can take a uniformly bounded sequence $l_k \in C_0^{\infty}(D_1)$ such that $\l_k \to l$ in $\cE_1$ and $\tilde{l}_k\to \tilde{l}$ q.e. by Theorem 4.4.3 of \cite{FOT}.
By the argument above, it holds that
\[
- \mathcal{E}(l_k \mathcal{X} \circ u, \bv_2)= \mathcal{E}(l_k \II(\mathcal{X}\circ u, \cH u),v)
\]
for each $k$. By letting $k \to \infty$, we have
\[
- \mathcal{E}(l \mathcal{X} \circ u, \bv_2)= \mathcal{E}(l \II(\mathcal{X}\circ u, \cH u),v)
\]
since $l_k w \to l w$ in $\cE_1$ for any $w \in \cF \cap L^{\infty}(E)$ due to the Beurling-Deny decomposition.
Therefore,
\begin{align*}
\lim_{t \to 0}\frac{1}{t}\Ex_{l\bdm}\left[\int \langle \mathcal{X} \circ u(Z), dN^{[\bv_2]} \rangle_{\delta_{\mathbb{R}^d}} \right] &=- \mathcal{E}(l \mathcal{X} \circ u, \bv_2)\\
&= \mathcal{E}(l \II(\mathcal{X}\circ u, \cH u),v).
\end{align*}
This yields
\[
\int_0^t \langle \mathcal{X} \circ u(Z), dN^{[\bv_2]} \rangle_{\delta_{\mathbb{R}^d}} = -\int_0^t \langle \II(\mathcal{X}\circ u(Z), \cH u(Z)), dN^{[v]} \rangle_{\delta_{\mathbb{R}^d}}\ \text{for}\ t\leq \tau_{D_1}
\]
by Lemma 5.4.4 of \cite{FOT} and the proof of Theorem 2.2 of \cite{Nakao}.
\end{proof}
Next, we prepare an extension of a function $\bdf \in C^{\infty}(TM)$ to the whole space $\mathbb{R}^d \times \mathbb{R}^d$.
Since $(TM,g^{\bcl})$ is a pseudo-Riemannian submanifold, we have the decomposition
\[
T(\mathbb{R}^d\times \mathbb{R}^d)=TTM \oplus TTM^{\perp}
\]
with respect to the metric $\delta_{\mathbb{R}^d}^{\bcl}$ and we have the normal exponential map $\exp^{TTM^{\perp}} \colon \mathcal{U}_{TTM^{\perp}} \to \mathbb{R}^d \times \mathbb{R}^d$, where $\mathcal{U}_{TTM^{\perp}}$ is a neighborhood of $o(TTM)=\{ 0 \in T_{\bV}TM \mid \bV \in TM\}$ such that $\exp^{TTM^{\perp}}$ is a diffeomorphism into its image. We define $\bar{\bdf}\colon \exp^{TTM^{\perp}}(\mathcal{U}_{TTM^{\perp}})\to \mathbb{R}$ by
\begin{align}\label{extTM}
\bar{\bdf}=\bdf \circ \pi_{TTM^{\perp}}\circ \left( \exp^{TTM^{\perp}} \right)^{-1}.
\end{align}
We need the following lemma to obtain a further extension of $\bar{\bdf}$ to the whole space $\mathbb{R}^d \times \mathbb{R}^d$.
\begin{lem}\label{properTM}
Assume that the embedding $\iota \colon M \to \mathbb{R}^d$ is proper. Then the embedding $\iota_* \colon TM \to \mathbb{R}^d \times \mathbb{R}^d$ is also proper.
\end{lem}
\begin{proof}
Let $K \subset \mathbb{R}^d \times \mathbb{R}^d$ be compact. Then $\pi_1(K)\subset \mathbb{R}^d$ is compact. Since $\iota$ is proper by assumption, $\iota^{-1}(\pi_1(K))$ is also compact. Let
\[
R:=\max_{(a,b)\in K}|b|.
\]
Then we have
\[
\iota_*^{-1}(K)\subset \{ \bV \in TM \mid \pi_{TM}\bV \in \iota^{-1}(\pi_1(K)),\ |\bV| \leq R \}
\]
and the right-hand side is compact. Thus $\iota_*^{-1}(K)$ is also compact.
\end{proof}
\begin{rem}\label{extTM2}
Under the assumption that the embedding $\iota \colon M \to \mathbb{R}^d$ is proper, $\iota_*(TM)$ is closed in $\mathbb{R}^d \times \mathbb{R}^d$ by \lref{properTM}. Thus by taking a cut-off function on $\mathbb{R}^d \times \mathbb{R}^d$, we can extend $\bar{\bdf}$ to a function in $C^{\infty}(\mathbb{R}^d\times \mathbb{R}^d)$, which we continue to denote by $\bar{\bdf}$.
\end{rem}
The following lemma states the relation between the gradient on $TM$ with respect to $g^{\bcl}$ and that on $\mathbb{R}^d\times \mathbb{R}^d$ with respect to $\delta_{\mathbb{R}^d}^{\bcl}$. 
\begin{lem}
Assume that the embedding $\iota \colon M \to \mathbb{R}^d$ is proper.
Let $\bdf \in C^{\infty}(TM)$ and let $\bar{\bdf} \in C^{\infty}(\mathbb{R}^d \times \mathbb{R}^d)$ be the extension of $\bdf$ given by \eqref{extTM}. Then it holds that
\begin{align}
\nabla^{\bcl}\bdf = D^{\bcl}\bar{\bdf}\ \text{on}\ TM,
\end{align}
where $D^{\bcl}$ is the gradient operator on $\mathbb{R}^d \times \mathbb{R}^d$ with respect to the metric $\delta_{\mathbb{R}^d}^{\bcl}$.
\end{lem}
\begin{proof}
For any $\bV \in TM$ and $\wilde{\bU} \in T_{\bV}TM^{\perp}$,
\[
\bar{\bdf}(\exp^{TTM^{\perp}}t\wilde{\bU})=\bdf(\bV)
\]
for sufficiently small $t\in \mathbb{R}$. Thus by differentiating both sides, we have
\[
\langle D^{\bcl}\bar{\bdf}(\bV), \wilde{\bU}\rangle_{\delta_{\mathbb{R}^d}^{\bcl}}=0,
\]
which means $D^{\bcl}\bar{\bdf}(\bV) \in T_{\bV}TM$ and equals $\nabla^{\bcl} \bdf(\bV)$. 
\end{proof}

\begin{proof}[Proof of \tref{DiffHarmonic}]
To begin with, we show that $\bu(Z)^{\tau_{D_1}}$ is a $\Prob_z$-semimartingale for q.e. $z\in E$.
Let $D_2$ be a relatively compact open subset with $\overline{D_1} \subset D_2$, $\overline{D_2}\subset D$. Let $\phi \in C^{\infty}_0(D)$ be a nonnegative smooth function satisfying
\begin{align*}
\begin{cases}
&\phi = 1\ \text{on}\ \overline{D_2},\\
&0\leq \phi \leq 1\ \text{on}\ E.
\end{cases}
\end{align*}
Then $\phi \bu \in \cF^{D}(\mathbb{R}^d\times \mathbb{R}^d)$. For $\bdf \in C^{\infty}(TM)$, let $\bar{\bdf} \in C^{\infty}(\mathbb{R}^d \times \mathbb{R}^d)$ be the extension of $\bdf$ given by \eqref{extTM}. Since it holds that
\begin{align*}
\bdf(\bu(Z_{t\land \tau_{D_1}})) &=\bar{\bdf}(\bu(Z_{t\land \tau_{D_1}}))\\
& = \bar{\bdf}(\phi \bu(Z_{t\land \tau_{D_1}})) + \mathbf{1}_{\{ t\geq \tau_{D_1} \}}\{\bar{\bdf}(\bu(Z_{t\land \tau_{D_1}})) - \bar{\bdf}(\phi \bu(Z_{t\land \tau_{D_1}}))\},
\end{align*}
it suffices to show that $\bar{\bdf}(\phi \bu(Z))$ is a $\Prob_z$-semimartingale for q.e. $z\in E$. Here we can take the embedding $d\iota$ as a function $\bdf$ to verify the semimartingale property and hence, we do not need to care for the dependence of q.e. set on $\bdf$. We note that the $\cE$-harmonicity of $u$ yields that $u(Z)^{\tau_{D_1}}$ is a $\Prob_z$-semimartingale for q.e. $z \in E$ by Theorem 3.7 of \cite{Oka24}. In particular, both $N^{[\phi u],\tau_{D_1}}$ and $N^{[v],\tau_{D_1}}$ are processes of locally bounded variation.
By It\^o's formula for Dirichlet processes $\phi \bu(Z)$ on $\mathbb{R}^d\times \mathbb{R}^d$ (\cites{Nakao, Kuwae}), it holds that
\begin{align*}
\bar{\bdf} \circ (\phi \bu)(Z)-\bar{\bdf} \circ (\phi \bu)(Z_0)&=\int \langle D\bar{\bdf} \circ (\phi \bu)(Z), dM^{[\phi \bu]}\rangle_{\delta_{\mathbb{R}^d} \oplus \delta_{\mathbb{R}^d}}\\
&\h + \int \langle D\bar{\bdf} \circ (\phi \bu)(Z), dN^{[\phi \bu]}\rangle_{\delta_{\mathbb{R}^d} \oplus \delta_{\mathbb{R}^d}}\\
&\h + (\text{a process of locally bounded variation}),
\end{align*}
where $D\bar{\bdf}$ stands for the gradient of $\bar{\bdf}$ for the Euclidean metric $\delta_{\mathbb{R}^d} \oplus \delta_{\mathbb{R}^d}$ on $\mathbb{R}^d \times \mathbb{R}^d$.
Note that it holds that
\[
\int_0^{t\land \tau_{D_1}}\langle D\bar{\bdf} \circ (\phi \bu) (Z), dN^{[\phi \bu]}\rangle_{\delta_{\mathbb{R}^d} \oplus \delta_{\mathbb{R}^d}} = \int_0^{t\land \tau_{D_1}} \langle (\phi D^{\bcl}\bar{\bdf}\circ \bu)(Z), dN^{[\phi \bu]}\rangle_{\delta_{\mathbb{R}^d}^{\bcl}}
\]
by
\[
D \bar{\bdf} \circ (\phi \bu) = \phi D \bar{\bdf} \circ \bu \ \text{on}\ D_1
\]
and
\[
\langle D\bar{\bdf}, (\ba,\bb) \rangle_{\delta_{\mathbb{R}^d} \oplus \delta_{\mathbb{R}^d}} = \langle D^{\bcl}\bar{\bdf}, (\ba,\bb) \rangle_{\delta_{\mathbb{R}^d}^{\bcl}},\ \text{for}\ \ba,\bb \in \mathbb{R}^d.
\]
For $\mathcal{X}\in \mathfrak{X}(M)$, it holds that
\begin{align*}
\int^{\cdot \land \tau_{D_1}} \langle v_{\bu(Z)}(\mathcal{X}), dN^{[\phi \bu]}\rangle_{\delta_{\mathbb{R}^d}^{\bcl}}&=\int^{\cdot \land \tau_{D_1}} \langle \mathcal{X}\circ u, dN^{[\phi u]} \rangle_{\delta_{\mathbb{R}^d}}.
\end{align*}
Therefore, $\dis \int_0^{\cdot \land \tau_{D_1}} \langle v_{\bu(Z)}(\mathcal{X}), dN^{[\phi \bu]}\rangle_{\delta_{\mathbb{R}^d}^{\bcl}}$ is a  process of bounded variation. We also have
\begin{align*}
\int_0^{\cdot \land \tau_{D_1}} \langle \mathcal{X}^{\bcl}\circ \bu(Z), dN^{[\phi \bu]}\rangle_{\delta_{\mathbb{R}^d}^{\bcl}}&=\int_0^{\cdot \land \tau_{D_1}} \langle \mathcal{X}\circ u(Z), dN^{[\phi \bu_2]}\rangle_{\delta_{\mathbb{R}^d}}\\
&\h + \int_0^{\cdot \land \tau_{D_1}} \langle \II(\cH u, \mathcal{X}\circ u)(Z), dN^{[\phi u]} \rangle_{\delta_{\mathbb{R}^d}} \\
&\h + \int_0^{\cdot \land \tau_{D_1}} \langle \nabla_{\cH u(Z)}\mathcal{X} (u(Z)), dN^{[\phi u]} \rangle_{\delta_{\mathbb{R}^d}},
\end{align*}
where $\bu_2=\pi_2(\bu)$.
The last two terms are processes of bounded variation. We set
\begin{align*}
\bA^{(2)}_t&:=\mathbf{1}_{\{t\geq \tau_{D_1}\}}\{(1-\phi (Z_{\tau_{D_1}}))\bu_2(Z_{\tau_{D_1}})-(1-\phi (Z_{\tau_{D_1}-}))\bu_2(Z_{\tau_{D_1}-})\},\\
\bB^{(2)}_t&:=\int_0^{t\land \tau_{D_1}}\int_{E\, \backslash \, D_2}(1-\phi)\bu_2(z)\, N(Z_s,dz)dH_s.
\end{align*}
Then by \lref{hAB}, the process $\bA^{(2)}$ is of $\Prob_z$-integrable variation and $\bB^{(2)}$ is the dual predictable projection for q.e. $z\in E$. In addition, we have
\begin{align}
N^{[\bv_2]}_t=N^{[\phi \bu_2]}_t+\bB^{(2)}_t, \ t\leq \tau_{D_1},\ \Prob_z\text{-a.s. for q.e.}\ z\in E,\label{Nphiu}
\end{align}
where $\bv_2=\pi_2(\bv)$.
Indeed, by the Fukushima decomposition for $\phi \bu$ and $\bv$, we have
\begin{align*}
\phi \bu(Z)-\phi \bu(Z_0)&= M^{[\phi \bu]} + N^{[\phi \bu]},\\
\bu(Z)-\bu(Z_0)&=\bh(Z)-\bh(Z_0) + M^{[\bv]}+N^{[\bv]}.
\end{align*}
Since $\bu(Z_t)=\phi \bu(Z_t)$ for $t<\tau_{D_1}$, we have
\[
M^{[\phi \bu_2]}_t+N^{[\phi \bu_2]}_t = \bH^{(2)}_t + M^{[\bv_2]}_t+N^{[\bv_2]}_t
\]
for $t<\tau_{D_1}$, $\Prob_z$-a.s. for q.e. $z\in E$, where
\begin{align*}
\bH^{(2)}_t&=\bh_2(Z_t)-\bh_2(Z_0),\\
\bh_2&=\pi_2(\bh).
\end{align*}
By incorporating the jump at $t=\tau_{D_1}$, we have
\[
M^{[\phi \bu_2]}_t+N^{[\phi \bu_2]}_t + \bA^{(2)}_t = \bH^{(2)}_t + M^{[\bv_2]}_t+N^{[\bv_2]}_t
\]
for $t\leq \tau_{D_1}$, $\Prob_z$-a.s. for q.e. $z\in E$. Therefore, we have
\begin{align}
M^{[\phi \bu_2]}_{t\land \tau_{D_1}} - M^{[\bv_2]}_{t\land \tau_{D_1}}+ \bA^{(2)}_{t\land \tau_{D_1}}-\bB^{(2)}_{t\land \tau_{D_1}} -\bH^{(2)}_{t\land \tau_{D_1}} = N^{[\bv_2]}_{t\land \tau_{D_1}}-N^{[\phi \bu_2]}_{t\land \tau_{D_1}}-\bB^{(2)}_{t\land \tau_{D_1}}.\label{martzeroen}
\end{align}
Here the left-hand side of \eqref{martzeroen} is a martingale. In addition, since the right-hand side of \eqref{martzeroen} is continuous, the left-hand side is a continuous martingale. Therefore, by taking the continuous local martingale part, we obtain
\[
M^{[\phi \bu_2],c}_{t\land \tau_{D_1}} - M^{[\bv_2],c}_{t\land \tau_{D_1}} - \bH^{(2),c}_{t\land \tau_{D_1}} = N^{[\bv_2]}_{t\land \tau_{D_1}}-N^{[\phi \bu_2]}_{t\land \tau_{D_1}}-\bB^{(2)}_{t\land \tau_{D_1}}.
\]
The right-hand side is a CAF of zero quadratic variation for the part process on $D_1$. Hence it is $\bdm$-equivalent to zero.
Moreover, in the same way as Remark 2.6 of \cite{Oka24}, it is indistinguishable from zero under $\Prob_z$-a.s. for q.e. $z\in E$ by the uniqueness of CAF's. Consequently, the left-hand side is identically zero under $\Prob_z$-a.s. for q.e. $z \in E$.
Therefore, we obtain \eqref{Nphiu}.
Thus by \lref{Nbv}, we have
\begin{align*}
\int_0^{\cdot \land \tau_{D_1}} \langle \mathcal{X}\circ u(Z), dN^{[\phi \bu_2]}\rangle_{\delta_{\mathbb{R}^d}} &= \int_0^{\cdot \land \tau_{D_1}} \langle \mathcal{X}\circ u(Z), dN^{[\bv_2]}\rangle_{\delta_{\mathbb{R}^d}} - \int_0^{\cdot \land \tau_{D_1}} \langle \mathcal{X}\circ u(Z), d\bB^{(2)} \rangle_{\delta_{\mathbb{R}^d}} \\
&=-\int_0^{\cdot \land \tau_{D_1}} \langle \II(\mathcal{X}\circ u, \cH u)(Z), dN^{[v]}\rangle_{\delta_{\mathbb{R}^d}}\\
&\h - \int_0^{\cdot \land \tau_{D_1}} \langle \mathcal{X}\circ u(Z), d\bB^{(2)} \rangle_{\delta_{\mathbb{R}^d}}.
\end{align*}
Therefore, $\dis \int_0^{\cdot \land \tau_{D_1}} \langle \mathcal{X}^{\bcl}\circ \bu(Z), dN^{[\phi \bu]}\rangle_{\delta_{\mathbb{R}^d}^{\bcl}}$ is a process of bounded variation for each $\cX \in \mathfrak{X}(M)$. Since complete lifts and vertical lifts of vector fields span each tangent space of $TM$, we deduce that $\bar{\bdf}(\phi \bu(Z))^{\tau_{D_1}}$ is a $\Prob_z$-semimartingale for q.e. $z\in E$. This yields that $\bJ^{\tau_{D_1}}$ is a $TM$-valued $\Prob_z$-semimartingale for q.e. $z\in E$. In particular, the process $N^{[\bv], \tau_{D_1}}$ is locally of bounded variation.

Next, we show that $\bJ^{\tau_{D_1}}$ is a $\Prob_z$-$\gamma^{\bcl}$-martingale for q.e. $z$. Since $u(Z)^{\tau_{D_1}}$ is a $\Prob_z$-$\gamma$-martingale for q.e. $z\in E$, \cref{martveccondition} shows that it suffices to verify the condition for complete lifts of vector fields. For each $\cX \in \mathfrak{X}(M)$, we have
\begin{align*}
\int_0^{\cdot \land \tau_{D_1}} \langle \mathcal{X}^{\bcl}\circ \bu(Z), dN^{[\bv]}\rangle_{\delta_{\mathbb{R}^d}^{\bcl}}&=\int_0^{\cdot \land \tau_{D_1}} \langle \mathcal{X}\circ u(Z), dN^{[\bv_2]}\rangle_{\delta_{\mathbb{R}^d}}\\
&\h + \int_0^{\cdot \land \tau_{D_1}} \langle \II(\mathcal{X}\circ u, \cH u)(Z), dN^{[v]} \rangle_{\delta_{\mathbb{R}^d}} \\
&\h + \int_0^{\cdot \land \tau_{D_1}} \langle \left(\nabla_{\cH u(Z)}\mathcal{X}\right) (u(Z)), dN^{[v]} \rangle_{\delta_{\mathbb{R}^d}} \\
&=0,
\end{align*}
where the first two terms vanish by \lref{Nbv} and the last term also vanishes since the integrand process is $TM$-valued and $u(Z)^{\tau_{D_1}}$ is a $\gamma$-martingale.
Therefore, $\bJ^{\tau_{D_1}}$ is a $\Prob_z$-$\gamma^{\bcl}$-martingale for q.e. $z\in E$.
\end{proof}

\section{Examples}\label{Examples}
In this section, we show some examples of Dirichlet forms and vector fields satisfying condition \thetag{C}.
\subsection{Symmetric L\'evy processes on Euclidean space}
Let $\{ \nu_t\}_{t\geq 0}$ be a convolution semigroup of probability measures on $\mathbb{R}^m$ and define the translation-invariant transition function $\{p_t\}_{t\geq 0}$ by 
\begin{align*}
p_tu(z):=\int_{\mathbb{R}^m}u(z+w)\, \nu_t(dw).
\end{align*}
Then $\{p_t\}_{t\geq 0}$ is the transition function of a L\'evy process. In addition, if $\{p_t\}_{t\geq 0}$ is symmetric, then the associated Dirichlet form $(\cE,\cF)$ is symmetric and regular. This Dirichlet form $(\cE,\cF)$ satisfies \thetag{C} for vector fields $\cH=\partial_i$ ($i=1,\dots,m$). In particular, if the corresponding L\'evy process is a symmetric $\alpha$-stable process for some $\alpha \in (0,2)$, the notion of $\cE$-harmonic maps defined in \dref{defharmonic} agrees with that of $\frac{\alpha}{2}$-fractional harmonic maps in \cites{MPS21}. In this case, provided $u$ is locally bounded on an open set $D$, conditions \thetag{A} and \thetag{B} are equivalent to
\[
\left( 1 \land |z|^{-(m+\alpha)} \right)|u| \in L^1(\mathbb{R}^m)
\]
by Example 2.12 in \cite{Chen09}.
Moreover, the partial regularity result for sphere-valued $\frac{\alpha}{2}$-harmonic maps obtained in \cite{MPS21} shows that under some conditions, there exists a closed polar set $\mathbf{N}$ such that $u \in C^{\infty}(D \backslash \mathbf{N})$. In such cases, if the differential $\partial_i u$ is realized as a locally integrable function on the whole $\mathbb{R}^m$ and satisfies
\[
\left( 1 \land |z|^{-(m+\alpha)} \right)|\partial_i u| \in L^1(\mathbb{R}^m),
\]
then the harmonic map satisfies the assumptions of \tref{DiffHarmonic} with $D$ replaced by $D \backslash \mathbf{N}$.

\subsection{Subordinated Brownian motions on Riemannian manifolds}
Let $(E,h)$ be a complete connected Riemannian manifold and denote the Laplace-Beltrami operator on $(E,h)$ by $\Delta_E$. Let $\bdm$ be the Riemannian volume measure on $(E,h)$. Let $(\cE,\cF)$ be a Dirichlet form on $L^2(E;\bdm)$ defined by
\begin{align*}
\begin{cases}
\cE(u,v)&=\frac{1}{2}\int_E h(\nabla^E u, \nabla^E v)\, d\bdm,\\
\cF&=\overline{C_0^{\infty}(E)}^{\cE_1},
\end{cases}
\end{align*}
where $\nabla^E$ is the gradient operator on $(E,h)$.
Then the Markov semigroup $\{T_t\}_{t\geq 0}$ on $L^2(E;\bdm)$ corresponding to $(\cE,\cF)$ is the heat semigroup and the associated Markov process is Brownian motion. We assume that Brownian motion on $(E,h)$ is conservative. In addition, we suppose that the generator $\frac{1}{2}\Delta$ of $\{T_t\}_{t\geq 0}$ on $L^2(E;\bdm)$ has $C_0^{\infty}(E)$ as its operator core.
For $\beta \in (0,1)$, the Bernstein function is given by
\begin{align*}
f^{(\beta)}(t)&=\int_0^{\infty}(1-e^{-ts})\, \mu^{(\beta)}(ds),\\
\mu^{(\beta)}(ds)&=\frac{\beta}{\Gamma(1-\beta)}\frac{ds}{s^{1+\beta}}.
\end{align*}
The subordinator associated with $f^{(\beta)}$ is defined by a convolution semigroup $\{\nu_t^{(\beta)}\}_{t\geq 0}$ of probability measures on $[0,\infty)$ characterized by the Laplace transform
\[
\int_0^{\infty}e^{-sx}\, \nu_t^{(\beta)}(ds)=e^{-tf^{(\beta)}(x)}.
\]
Let $\alpha:=2\beta$. The $\alpha$-subordinated semigroup $\{T_t^{(\alpha)}\}_{t\geq 0}$ on $L^2(E;\bdm)$ is defined by
\begin{align}\label{subsemi}
T^{(\alpha)}_tu=\int_0^{\infty} T_su \, \nu_t^{\left(\frac{\alpha}{2}\right)}(ds).
\end{align}
Then $\{T^{(\alpha)}_t\}_{t\geq 0}$ is a symmetric Markov semigroup on $L^2(E;\bdm)$ and its corresponding Dirichlet form $(\cE^{(\alpha)},\cF^{(\alpha)})$ is regular by \cite{AR05} and conservative by \eqref{subsemi}. Moreover, under the assumption that $C_0^{\infty}(E)$ is an operator core of $\frac{1}{2}\Delta$, $C_0^{\infty}(E)$ is also an operator core of the generator $\cL^{(\alpha)}$ of $\{T^{(\alpha)}_t\}_{t\geq 0}$ (e.g. \cite{AR05}).
The Dirichlet form $(\cE^{(\alpha)},\cF^{(\alpha)})$ satisfies \thetag{C} for each Killing vector field. Indeed, by the representation \eqref{subsemi}, we have
\[
[T_t^{(\alpha)},\cH]\phi=0
\]
for $\phi \in C_0^{\infty}(E)$. Consequently, we have
\[
[\cL^{(\alpha)},\cH]\phi=0.
\]

\subsection{Isotropic L\'evy processes on Riemannian manifolds}\label{Levymfd}
Let $(E,h)$ be an $m$-dimensional compact connected Riemannian manifold. Let $\nu$ be a rotationally invariant measure on $\mathbb{R}_0^m:=\mathbb{R}^m \backslash \{0\}$ satisfying
\[
\int_{\mathbb{R}^m_0}1 \land |a|^2 \, d\nu(a)<\infty.
\]
For each $z\in E$, we define the measure $\nu_z$ on $T_zE$ by
\[
\nu_z=\be_{z\sharp}\nu,
\]
where $\be_z \colon \mathbb{R}^m \to T_zE$ is a linear isometry and $\sharp$ stands for the push-forward measure. Then since $\nu$ is rotationally invariant, $\nu_z$ does not depend on the choice of $\be_z$. We consider the linear operator
\[
\cL u(z):=\frac{\sigma}{2} \Delta_E u(z)+\int_{T_zE} \left\{ u(\exp_z \bV)-u(z)-\mathbf{1}_{\{|\bV| <1 \}}h(\nabla u(z), \bV) \right\}\, d\nu_z(\bV)
\]
for $u \in C^{\infty}(E)$ and $\sigma \geq 0$. Then by Theorem 3.1 in \cite{AppleEst00}, there exists a Feller process $\{Z_t\}$ on $E$ whose generator is given by $\cL$. The process $Z$ is called an isotropic L\'evy process on $E$. We denote the transition function of $Z$ by $\{p_t\}_{t \geq 0}$.
Then by Theorem 4.2 of \cite{Apple21}, $\{p_t\}_{t\geq 0}$ extends to the strongly continuous self-adjoint Markov semigroup on $L^2(E;\bdm)$.
In particular, if we let $(\cE,\cF)$ be a Dirichlet form associated with $\{Z_t\}_{t\geq 0}$, then $(\cE,\cF)$ is regular. Moreover, the proposition below shows that this Dirichlet form $(\cE,\cF)$ satisfies condition \thetag{C-2} for every Killing vector field.
\begin{prop}
For any $u \in C^{\infty}(E)$ and any Killing vector field $\cH$ on $E$, it holds that
\[
[\cL,\cH]u=0.
\] 
\end{prop}
\begin{proof}
Since $\cH$ is Killing, $[\Delta, \cH]u=0$ for all $u \in C^{\infty}(E)$. Let $\Phi_{\ep} \colon E \to E$ be the flow generated by $\cH$. Then $\Phi_{\ep}$ is an isometry on $E$. Thus we have
\begin{align*}
\Phi_{\ep}(\exp_z \bV)&=\exp_{\Phi_{\ep}(z)}\Phi_{\ep *}(\bV),\\
\nabla (u\circ \Phi_{\ep})(z)&=\Phi_{-\ep *}(\nabla u (\Phi_{\ep}(z))).
\end{align*}
Therefore, the jump part of $\cL(\Phi_{\ep}^*u)$ is given by
\begin{align*}
&\int_{T_zE}\left\{ u\circ \Phi_{\ep}(\exp_z \bV)- u\circ \Phi_{\ep}(z)-\mathbf{1}_{\{ |\bV|<1 \}}h_z(\nabla (u\circ \Phi_{\ep})(z), \bV) \right\}\, d\nu_z(\bV)\\
&=\int_{T_zE}\left\{ u(\exp_{\Phi_{\ep}(z)} \Phi_{\ep *}\bV)- u\circ \Phi_{\ep}(z)-\mathbf{1}_{\{ |\bV|<1 \}}h_{\Phi_{\ep}(z)}((\nabla u)(\Phi_{\ep}(z)), \Phi_{\ep *}\bV) \right\}\, d\nu_z(\bV)\\
&=\int_{T_{\Phi_{\ep}(z)}E}\left\{ u(\exp_{\Phi_{\ep}(z)} \bV)- u\circ \Phi_{\ep}(z)-\mathbf{1}_{\{ |\bV|<1 \}}h_{\Phi_{\ep}(z)}((\nabla u)(\Phi_{\ep}(z)), \bV) \right\}\, d\nu_{\Phi_{\ep}(z)}(\bV).
\end{align*}
Here we used $\Phi_{\ep * \sharp}\nu_z=\nu_{\Phi_{\ep}(z)}$ in the third equality. Thus we get
\begin{align}
\cL (\Phi_{\ep}^*u)=\Phi_{\ep}^*(\cL u).\label{commuflow}
\end{align}
We can easily check that $\cL \colon C^2(E) \to C(E)$ is continuous. Since
\[
\frac{\Phi_{\ep}^*u-u}{\ep} \to \cH u\ \text{in}\ C^2(E)
\]
for $u\in C^{\infty}(E)$, the limit
\[
\lim_{\ep \to 0}\frac{\cL (\Phi_{\ep}^*u) - \cL u}{\ep}
\]
exists and equals $\cL \cH u$. Therefore, by \eqref{commuflow}, the directional derivative
\[
\cH \cL u=\lim_{\ep \to 0}\frac{\Phi_{\ep}^*(\cL u)-\cL u}{\ep}
\]
also exists and equals $\cL \cH u$.
\end{proof}
The next proposition shows that the generator of the form $(\cE,\cF)$ has $C^{\infty}(E)$ as its operator core if it has nonzero Brownian component. Consequently, $(\cE,\cF)$ satisfies condition \thetag{C} in this case.
\begin{prop}
Assume $\sigma>0$. Then $C^{\infty}(E)$ is an operator core of $\cL$ on $L^2(E;\bdm)$.
\end{prop}
\begin{proof}
By Theorem 4.2 of \cite{Apple21}, $\cL$ is self-adjoint. We decompose $\cL$ as $\cL=\cL_0+\cL_J$,
\begin{align*}
\cL_0 u(z)&:=\frac{\sigma}{2} \Delta_E u(z),\\
\cL_J u(z)&:=\int_{T_zE} \left\{ u(\exp_z \bV)-u(z)-\mathbf{1}_{\{|\bV| <1 \}}h(\nabla u(z), \bV) \right\}\, d\nu_z(\bV).
\end{align*}
Since $E$ is compact, $C^{\infty}(E)$ is an operator core of $\cL_0$ on $L^2(E;\bdm)$ and $\cL_0$ is self-adjoint.
We apply the Kato-Rellich theorem (Chapter V, Theorem 4.3, 4.4 in \cite{Kato}) to $\cL_0+\cL_J$.
In order to do that, we need to get the $L^2$ estimate
\begin{align}
\| \cL_Ju \|_{L^2(E;\bdm)} \leq c_1 \| \cL_0 u \|_{L^2(E;\bdm)} + c_2 \| u \|_{L^2(E;\bdm)}\label{L0bound}
\end{align}
for any $u \in C^{\infty}(E)$ and some $c_1 \in [0,1)$ and $c_2 \geq 0$ independent of $u$. Let $r>0$ and set
\begin{align*}
\cL^{(1)}_J u(z)&:=\int_{|\bV|\geq r} \left\{ u(\exp_z \bV)-u(z) \right\}\, d\nu_z(\bV),\\
\cL^{(2)}_J u(z)&:=\int_{0< |\bV| <r} \left\{ u(\exp_z \bV)-u(z)- h(\nabla u(z), \bV) \right\}\, d\nu_z(\bV).
\end{align*}
Then since $\nu$ is rotationally invariant and $\nu(\{a\mid r\leq |a| <1\})$ is finite, $\cL_J=\cL_J^{(1)}+\cL_J^{(2)}$.
Let $\pi_{\cO(E)}\colon \cO(E) \to E$ be an orthonormal frame bundle and denote the Liouville measure on $\cO(E)$ by $\wilde{\bdm}$. Suppose that $\wilde{\bdm}$ is normalized in such a way that $\pi_{\cO(E)\sharp}\wilde{\bdm}=\bdm$. Let $\wilde{\cA}_k$ ($k=1,\dots,m$) be canonical horizontal vector fields on $\cO(E)$.
Then for any $a\in \mathbb{R}^m$, the horizontal vector field $a^k\wilde{\cA}_k$ is complete and $\wilde{\bdm}$ is invariant with respect to the flow $\mathrm{Exp}(a^k \wilde{\cA}_k)$ (c.f. \cite{Apple21}).
For $\xi \in \cO_z(E)$, it holds that
\begin{align*}
\pi_{\cO(E)}^*(\cL_J^{(1)}u)(\xi)&= \int_{|\bV|\geq r} \left\{ u(\exp_{\pi_{\cO(E)}\xi} \bV) - u(\pi_{\cO(E)}\xi) \right\}\, d\nu_{\pi_{\cO(E)}\xi}(\bV)\\
&=\int_{|\bV|\geq r} \left\{\pi_{\cO(E)}^*u(\mathrm{Exp}(\xi^{-1}\bV)^{k}\wilde{\cA}_k(\xi)) -\pi_{\cO(E)}^*u(\xi) \right\}\, d\xi_{\sharp}\nu (\bV)\\
&=\int_{|a|\geq r} \left\{\pi_{\cO(E)}^*u(\mathrm{Exp}(a^{k}\wilde{\cA}_k)(\xi)) -\pi_{\cO(E)}^*u(\xi) \right\}\, d\nu (a).
\end{align*}
Therefore,
\begin{align*}
\int_E &\left| \cL_J^{(1)}u(z) \right|^2 \, \bdm(dz)=\int_{\cO(E)} \left| \pi_{\cO(E)}^*(\cL_J^{(1)}u)(\xi) \right|^2 \, \wilde{\bdm}(d\xi)\\
&\leq \nu(\{ |a|\geq r \}) \int_{\cO(E)} \int_{|a|\geq r}\left| \pi_{\cO(E)}^*u(\mathrm{Exp}(a^{k}\wilde{\cA}_k)(\xi)) -\pi_{\cO(E)}^*u(\xi) \right|^2\, d\nu(a) \, \wilde{\bdm}(d\xi)\\
&\leq 4\nu(\{ |a|\geq r \})^2\int_{\cO(E)} |\pi_{\cO(E)}^*u(\xi)|^2 \, \wilde{\bdm}(d\xi)\\
&=4\nu(\{ |a|\geq r \})^2\int_{E} |u(z)|^2 \, \bdm(dz),
\end{align*}
where we used the invariance of $\wilde{\bdm}$.
In the same way, it holds that
\begin{align*}
&\pi_{\cO(E)}^*(\cL_J^{(2)}u)(\xi)\\
&=\int_{\{0<|a|< r\}} \left\{\pi_{\cO(E)}^*u(\mathrm{Exp}(a^{k}\wilde{\cA}_k)(\xi)) -\pi_{\cO(E)}^*u(\xi) - a^k\wilde{\cA}_k \pi_{\cO(E)}^*u(\xi) \right\}\, d\nu (a)\\
&=\int_{\{0<|a|< r\}} \int_0^1 (1-s)a^k a^l\wilde{\cA}_k \wilde{\cA}_l \pi_{\cO(E)}^*u(\mathrm{Exp}(sa^i\wilde{\cA}_i)(\xi)) \, ds\, d\nu (a)
\end{align*}
by Taylor's theorem.
Therefore, we have
\begin{align*}
&\int_E \left| \cL_J^{(2)}u(z) \right|^2 \, \bdm(dz)\\
&\leq \frac{1}{2} \left(\int_{0<|a|<r}|a|^2\, \nu(da) \right) \int_0^1(1-s) ds \int_{\cO(E)} |a^ka^l\wilde{\cA}_k \wilde{\cA}_l \pi_{\cO(E)}^*u(\xi)|^2 \, \wilde{\bdm}(\xi)\, \frac{1}{|a|^2}\nu(da)\\
&\leq \frac{1}{4}\left(\int_{0<|a|<r}|a|^2\, \nu(da) \right)^2 \|\mathrm{Hess}\, u\|_{L^2(E;\bdm)}^2,
\end{align*}
where we used $a^ka^l\wilde{\cA}_k \wilde{\cA}_l \pi_{\cO(E)}^*u(\xi)=\mathrm{Hess}\, u(z)(\xi a, \xi a)$ and the invariance of $\wilde{\bdm}$.
On the other hand, by the Bochner formula, it holds that
\[
|\mathrm{Hess}\, u|^2=\frac{1}{2}\Delta |\nabla u|^2-h(\nabla \Delta u, \nabla u)-\mathrm{Ric}(\nabla u, \nabla u).
\]
By integrating both sides, we get
\[
\| \mathrm{Hess}\, u \|_{L^2(E;\bdm)}^2 \leq \| \Delta u \|_{L^2(E;\bdm)}^2 + C\| \nabla u \|_{L^2(E;\bdm)}^2.
\]
for some constant $C>0$. Since
\begin{align*}
\|\nabla u \|_{L^2(E;\bdm)}^2&=-\int_E u \Delta u\, d\bdm\\
&\leq \|u \|_{L^2(E;\bdm)}\|\Delta u \|_{L^2(E;\bdm)}\\
&\leq \frac{\|u \|_{L^2(E;\bdm)}^2 + \|\Delta u \|_{L^2(E;\bdm)}^2}{2},
\end{align*}
we have
\[
\|\mathrm{Hess}\, u \|_{L^2(E;\bdm)}^2 \leq C \left( \|\Delta u \|_{L^2(E;\bdm)}^2 + \| u \|_{L^2(E;\bdm)}^2\right)
\]
for some $C>0$. Therefore, by choosing $r>0$ sufficiently small, we conclude that for any $\ep \in (0,1)$, there exists $C_{\ep}>0$ such that for any $u\in C^{\infty}(E)$,
\[
\|\cL_Ju \|_{L^2(E;\bdm)} \leq \ep \| \cL_0 u \|_{L^2(E;\bdm)} + C_{\ep}\| u \|_{L^2(E;\bdm)}.
\]
Thus we obtain \eqref{L0bound}. The estimate \eqref{L0bound} and the Kato-Rellich theorem show that $C^{\infty}(E)$ is an operator core of $\cL$.
\end{proof}

\section{Stochastic parallel transport with projected jumps along \cl semimartingales on manifolds}\label{SectionParallel}
In order to obtain a differential formula for harmonic maps, we introduce the parallel transport of tangent vectors with projected jumps along \cl semimartingales on Riemannian submanifolds. Let $X$ be a \cl semimartingale on $M$. Let $\bV_0$ be a $TM$-valued $\cF_0$-measurable random variable with $\pi_{TM}\bV_0=X_0$. We consider the following SDE on $\mathbb{R}^d$:
\begin{align}
\cV_t - \bV_0&= \int_0^t(\partial_i \Pi_{X_{s-}})\cV_{s-} \, dX^i_s\nonumber \\
&\h +\frac{1}{2} \int_0^t \left( \partial_i \partial_j \Pi_{X_{s-}}+\partial_i \Pi_{X_{s-}}\partial_j\Pi_{X_{s-}}\right) \cV_{s-} \, d[X^i,X^j]^c_s\nonumber \\
&\h + \sum_{0<s \leq t}\left( \Pi_{X_s}\cV_{s-} - \cV_{s-} - (\partial_i\Pi_{X_{s-}}\cV_{s-}) \Delta X^i_s \right).\label{parallel}
\end{align}
Here we extend the map $x\mapsto \Pi_x$ smoothly to a tubular neighborhood of $M$ by composing it with the normal projection.
\begin{prop}\label{SDEtransport}
SDE \eqref{parallel} admits a unique global $TM$-valued solution $\cV$ satisfying
\[
\cV_t\in T_{X_t}M,\ |\cV_t| \leq |\bV_0|,\ t\geq 0.
\]
\end{prop}
\begin{proof}
We define $R \colon M \times M \to \mathbb{R}^d \otimes \mathbb{R}^d$ by
\[
R(x,y):=\Pi_y-\Pi_x-\partial_i \Pi_x(y^i-x^i).
\]
Then the purely discontinuous adapted process $\cR$ defined by
\[
\cR_t:=\sum_{0<s\leq t}R(X_{s-},X_s)
\]
is well-defined and locally of finite variation. We further define the $\mathbb{R}^d \otimes \mathbb{R}^d$-valued semimartingale $\cK$ by
\begin{align*}
\cK_t&:=\int_0^t \partial_i \Pi_{X_{s-}}\, dX^i_s\\
&\h + \frac{1}{2}\int_0^t (\partial_i \partial_j \Pi_{X_{s-}}+ \partial_i \Pi_{X_{s-}}\partial_j \Pi_{X_{s-}})\, d[X^i,X^j]^c_s + \cR_t.
\end{align*}
We consider the SDE
\begin{align}\label{parallel2}
\cV_t=\bV_0+\int_0^t (d\cK_s)\, \cV_{s-}.
\end{align}
Since \eqref{parallel2} is a linear SDE, it admits a unique global solution $\cV$. Next, we show that $\cV_t \in T_{X_t}M$. Let $\cW_t:=\Pi^{\perp}_{X_t}\cV_t$. Then $\cW_0=\Pi^{\perp}_{X_0}\bV_0=0$. 
We derive the SDE that $\cW$ satisfies. It holds that
\begin{align*}
d\cW_t&=-(d\Pi_{X_t})\, \cV_{t-}+\Pi^{\perp}_{X_{t-}}\, d\cV_t-d[\Pi_X,\cV]_t,\\
(d\Pi_{X_t})\, \cV_{t-}&=(\partial_i\Pi_{X_{t-}})\cV_{t-}\, dX^i_t+\frac{1}{2}(\partial_i\partial_j \Pi_{X_{t-}})\cV_{t-}\, d[X^i,X^j]^c_t\\
&\h +\{ \Pi_{X_t}\cV_{t-}- \Pi_{X_{t-}}\cV_{t-} - (\partial_i \Pi_{X_{t-}})\cV_{t-} \Delta X^i_t \},\\
\Pi^{\perp}_{X_{t-}}\, d\cV_t&=\Pi^{\perp}_{X_{t-}} (\partial_i \Pi_{X_{t-}})\cV_{t-}\, dX^i_t+ \frac{1}{2}\Pi^{\perp}_{X_{t-}}(\partial_i \partial_j \Pi_{X_{t-}}+ \partial_i \Pi_{X_{t-}}\partial_j \Pi_{X_{t-}})\cV_{t-}\, d[X^i,X^j]^c_t\\
&\h + \Pi^{\perp}_{X_{t-}}\{ \Pi_{X_t}\cV_{t-}-\Pi_{X_{t-}}\cV_{t-}
-(\partial_i \Pi_{X_{t-}})\cV_{t-} \Delta X^i_t \},\\
d[\Pi_X,\cV]_t&=d[\Pi_X,\cV]^c_t+\Delta \Pi_{X_t} \Delta \cV_t\\
&= (\partial_i \Pi_{X_{t-}})(\partial_j \Pi_{X_{t-}})\cV_{t-}\, d[X^i,X^j]^c_t + (\Pi_{X_t} \cV_t- \Pi_{X_t}\cV_{t-} - \Pi_{X_{t-}}\cV_t + \Pi_{X_{t-}} \cV_t).
\end{align*}
Since $x\mapsto \Pi_x$ is extended to a tubular neighborhood of $M$ by the composition with the normal projection, it satisfies $\Pi_x^2=\Pi_x$. Therefore, it also satisfies
\begin{align}
\partial_i \Pi_x&=(\partial_i \Pi_x)\Pi_x+\Pi_x(\partial_i \Pi_x),\label{partial1Pi}\\
\partial_i \partial_j \Pi_x&=\Pi_x (\partial_i \partial_j \Pi_x) + (\partial_i \Pi_x) (\partial_j \Pi_x) + (\partial_j \Pi_x) (\partial_i \Pi_x) + (\partial_i \partial_j \Pi_x) \Pi_x.\label{partial2Pi}
\end{align}
By \eqref{partial1Pi} and \eqref{partial2Pi}, we have
\begin{align}
\cW_t&=-\int_0^t(\partial_i \Pi_{X_{s-}})\cW_{s-}\, dX^i_s+\frac{1}{2}\int_0^t \left(-\partial_i \partial_j \Pi_{X_{s-}} + (\partial_i \Pi_{X_{s-}}) (\partial_j \Pi_{X_{s-}}) \right)\cW_{s-}\, d[X^i,X^j]^c_s\nonumber \\
&\h -\sum_{0<s\leq t}\{ \Pi_{X_s} - \Pi_{X_{s-}} - (\partial_i \Pi_{X_{s-}})\Delta X_s^i \}\cW_{s-}.\label{SDEW}
\end{align}
Since $0$ is the unique solution to \eqref{SDEW}, we have $\cW_t=0$ for $t\geq 0$ a.s., which yields $\cV_t \in T_{X_t}M$. Moreover, due to $\Pi_{X_{s-}}\cV_{s-}=\cV_{s-}$, SDE \eqref{parallel} is equivalent to \eqref{parallel2}. Finally, we show that $|\cV_t| \leq |\bV_0|$ for $t\geq 0$.
Let $\{\ep_i\}_{i=1}^d$ be the canonical basis on $\mathbb{R}^d$. Then
\begin{align}\label{partialPi}
(\partial_i \Pi_{X_-})\cV_-=\II(\Pi_{X_-}\ep_i, \cV_-) \perp T_{X_-}M.
\end{align}
It holds that
\begin{align*}
d |\cV_t|^2&=2 \langle \cV_{t-}, d\cV_t \rangle_{\delta_{\mathbb{R}^d}}+d[\cV, \cV]_t,\\
2 \langle \cV_{t-}, d\cV_t \rangle_{\delta_{\mathbb{R}^d}}&=2\langle \cV_{t-}, (\partial_i \Pi_{X_t-})\cV_{t-} \rangle_{\delta_{\mathbb{R}^d}} \, dX^i_t\\
&\h +\langle \cV_{t-}, \partial_i \partial_j\Pi_{X_{t-}}+(\partial_i \Pi_{X_{t-}}) (\partial_j \Pi_{X_{t-}})\cV_{t-} \rangle_{\delta_{\mathbb{R}^d}}\, d[X^i,X^j]^c_t\\
&\h - 2 \langle \cV_{t-}, \Pi^{\perp}_{X_{t}}\cV_{t-} \rangle_{\delta_{\mathbb{R}^d}},\\
d[\cV,\cV]_t&=d[\cV,\cV]^c_t+\langle | \Delta \cV_t |^2\\
&=\langle (\partial_i \Pi_{X_{t-}})\cV_{t-}, (\partial_j \Pi_{X_{t-}})\cV_{t-} \rangle_{\delta_{\mathbb{R}^d}}\, d[X^i,X^j]^c_t + |\Pi^{\perp}_{X_t}\cV_{t-}|^2,
\end{align*}
Using \eqref{partialPi}, \eqref{partial2Pi} and
\begin{align*}
\langle \cV_{t-}, \Pi^{\perp}_{X_{t}}\cV_{t-} \rangle_{\delta_{\mathbb{R}^d}}=| \Pi^{\perp}_{X_t}\cV_{t-}|^2,
\end{align*}
we have
\[
|\cV_t|^2=|\bV_0|^2-\sum_{0<s\leq t}| \Pi^{\perp}_{X_s}\cV_{s-}|^2.
\]
In particular, this means $|\cV_t| \leq |\bV_0|$.
\end{proof}
\begin{rem}
If $X$ is continuous, then the solution of \eqref{parallel} is the stochastic parallel transport along $X$ with respect to the Levi-Civita connection $\nabla$.
\end{rem}

\begin{lem}\label{parallelFormula}
Let $\cV$ be the solution to \eqref{parallel}. Let $\bJ=(X,J)$ be a $TM$-valued semimartingale with $\pi_{TM}\bJ=X$. Then
\begin{align*}
\langle \cV_t, J_t \rangle_{\delta_{\mathbb{R}^d}} - \langle \bV_0, J_0 \rangle_{\delta_{\mathbb{R}^d}}&=\int_0^t\langle \cV_{s-},dJ_s \rangle_{\delta_{\mathbb{R}^d}}+\frac{1}{2}\int_0^t \langle \II(\cV_{s-},dX^c_s), \II(J_{s-},dX^c_s) \rangle_{\delta_{\mathbb{R}^d}}.
\end{align*}
\end{lem}
\begin{proof}
By integration by parts,
\[
d\langle \cV,J \rangle_{\delta_{\mathbb{R}^d}}=\langle \cV_-,dJ \rangle_{\delta_{\mathbb{R}^d}}+ \langle J_-, d\cV \rangle_{\delta_{\mathbb{R}^d}} + d[\cV,J].
\]
Let $\{\ep_i\}_{i=1}^d$ be the canonical basis on $\mathbb{R}^d$. By \eqref{partialPi} and \eqref{partial2Pi}, we have
\begin{align*}
\langle J_-,(\partial_i \Pi_{X_-})\cV_- \rangle_{\delta_{\mathbb{R}^d}}&=0,\\
\langle J_-, (\partial_i \partial_j \Pi_{X_-} + (\partial_i \Pi_{X_-}) (\partial_j \Pi_{X_-})) \cV_- \rangle_{\delta_{\mathbb{R}^d}} &= -\langle J_-, (\partial_j \Pi_{X_-}) (\partial_i \Pi_{X_-}) \cV_- \rangle_{\delta_{\mathbb{R}^d}}\\
&= -\langle (\partial_j \Pi_{X_-}) J_-, (\partial_i \Pi_{X_-}) \cV_- \rangle_{\delta_{\mathbb{R}^d}}\\
&= -\langle \II(\Pi_{X_-}\ep_j,J_-), \II(\Pi_{X_-}\ep_i, \cV_-) \rangle_{\delta_{\mathbb{R}^d}}.
\end{align*}
We also have
\begin{align*}
&\langle J_-, \Pi_{X}\cV_- - \cV_- - (\partial_i\Pi_{X_-}\cV_-) \Delta X^i \rangle_{\delta_{\mathbb{R}^d}}\\
&=\langle J_-, \Pi_{X}\cV_- - \cV_- \rangle_{\delta_{\mathbb{R}^d}}
\end{align*}
Therefore, by \eqref{parallel}, we have
\begin{align*}
\langle J_-, d\cV \rangle_{\delta_{\mathbb{R}^d}}=-\frac{1}{2}\langle \II(J_-, d^{\gamma}X^c), \II(\cV_-, d^{\gamma}X^c) \rangle_{\delta_{\mathbb{R}^d}} + \langle J_-, \Pi_{X}\cV_- - \cV_- \rangle_{\delta_{\mathbb{R}^d}}.
\end{align*}
On the other hand, taking the continuous local martingale parts of $\Pi^{\perp}_{X_t}J_t=0$, we obtain
\[
\Pi^{\perp}_{X_{t-}}dJ^c_t-(\partial_j \Pi_{X_{t-}})J_{t-}dX^{j,c}_t=0.
\]
Thus, it holds that
\begin{align*}
[J,\cV]^c&=\int \langle (\partial_i \Pi_{X_{t-}})\cV_{t-}, \ep_j \rangle_{\delta_{\mathbb{R}^d}}\, d[X^i,J^j]^c\\
&=\int \langle (\partial_i \Pi_{X_{t-}})\cV_{t-}, \Pi_{X_{t-}}^{\perp}\ep_j \rangle_{\delta_{\mathbb{R}^d}}\, d[X^i,J^j]^c\\
&=\left[\int (\partial_i \Pi_{X_-})\cV_- \, dX, \int \Pi^{\perp}_{X_-}dJ \right]^c\\
&=\left[ \int (\partial_i \Pi_{X_-})\cV_- \, dX^i, \int (\partial_j \Pi_{X_-})J_- \, dX^j \right]^c\\
&=\int \langle \II(\cV_-, d^{\gamma}X^c), \II(\bJ_-, d^{\gamma}X^c) \rangle_{\delta_{\mathbb{R}^d}},\\
[J,\cV]^d&=\sum_{0<s\leq \cdot} \langle \Delta J_s, \Pi_{X_s}\cV_{s-}-\cV_{s-}\rangle_{\delta_{\mathbb{R}^d}}.
\end{align*}
Therefore,
\begin{align*}
\langle J_-, d\cV \rangle_{\delta_{\mathbb{R}^d}} + d[\cV,J]= \frac{1}{2}\langle \II(\bJ_-, d^{\gamma}X^c), \II(\cV_-, d^{\gamma}X^c) \rangle_{\delta_{\mathbb{R}^d}},
\end{align*}
where we used
\[
\langle J, \Pi_{X}\cV_- - \cV_- \rangle_{\delta_{\mathbb{R}^d}}=0.
\]
This completes the proof.
\end{proof}

\begin{thm}\label{MeanValue}
Under the same assumptions as \tref{DiffHarmonic}, fix a relatively compact open set $D_1 \subset D$ with $\overline{D_1}\subset D$ and set $X:=u(Z)$, $J=\cH u(Z)$, $\bJ:=(X,J)$. Let $\cV$ be the solution of \eqref{parallel} with initial value $\bV_0 \in T_{X_0}M$.
Then, the process
\begin{align}
&g(\cV_{t\land \tau_{D_1}}, J_{t\land \tau_{D_1}})-g(\bV_0,J_0)+\frac{1}{2}\int_0^{t\land \tau_{D_1}} g(\cV_s, R_M^{\nabla}(J_{s},dX^{\tau_{D_1},c}_s)dX^{\tau_{D_1},c}_s ) \nonumber \\
&+ \sum_{0<s\leq t \land \tau_{D_1}} \langle \II( \cV_{s\land \tau_{D_1}-} ,J_{s \land \tau_{D_1}-}), \gamma^{\perp}(X_{s \land \tau_{D_1}-},X_{s\land \tau_{D_1}}) \rangle_{\delta_{\mathbb{R}^d}} \label{Difflocalmartingale}
\end{align}
is a local martingale starting from $0$, where $R_M^{\nabla} \in \Omega^2(\mathrm{End}(TM))$ is the curvature tensor on $M$ associated with the Levi-Civita connection $\nabla$ defined by
\[
R_M^{\nabla}(\cX, \cY) \cZ = \nabla_{\cX}\nabla_{\cY} \cZ - \nabla_{\cY}\nabla_{\cX} \cZ - \nabla_{[\cX,\cY]} \cZ,\ \cX, \cY, \cZ \in \mathfrak{X}(M).
\]
\end{thm}
\begin{proof}
By \lref{parallelFormula}, we have
\begin{align*}
\langle \cV_{t\land \tau_{D_1}}, J_{t\land \tau_{D_1}} \rangle_{\delta_{\mathbb{R}^d}} - \langle \bV_0, J_0 \rangle_{\delta_{\mathbb{R}^d}}&=\int_0^{t\land \tau_{D_1}}\langle \cV_{s-},dJ_s \rangle_{\delta_{\mathbb{R}^d}}\\
&\h +\frac{1}{2}\int_0^{t\land \tau_{D_1}} \langle \II(\cV_{s-},dX^c_s), \II(J_{s-},dX^c_s) \rangle_{\delta_{\mathbb{R}^d}}.
\end{align*}
On the other hand, since $\bJ^{\tau_{D_1}}$ is a $\gamma^{\bcl}$-martingale by \tref{DiffHarmonic}, the process
\begin{align*}
\int^{\cdot \land \tau_{D_1}}\lambda_{g^{\bcl}}(h^{\nabla}_{\bJ}(\cV_-))\, d^{\gamma^{\bcl}}\bJ=&\int^{\cdot \land \tau_{D_1}} \langle \cV_-, dJ \rangle_{\delta_{\mathbb{R}^d}} + \frac{1}{2}\int^{\cdot \land \tau_{D_1}} \langle \II(\cV_-,J_-),\II(dX^c,dX^c)\rangle_{\delta_{\mathbb{R}^d}}\\
&\h +\sum_{0<s\leq \cdot \land \tau_{D_1}} \langle \II( \cV_- , J_{s-}), \gamma^{\perp}(X_-,X) \rangle_{\delta_{\mathbb{R}^d}},
\end{align*}
is a local martingale by \cref{martveccondition}. Therefore, the process
\begin{align*}
\int^{\cdot \land \tau_{D_1}}\lambda_{g^{\bcl}}(h^{\nabla}_{\bJ}(\cV_-))\, d^{\gamma^{\bcl}}\bJ&= \langle \cV_{t\land \tau_{D_1}}, J_{t\land \tau_{D_1}} \rangle_{\delta_{\mathbb{R}^d}}- \langle \bV_0, J_0 \rangle_{\delta_{\mathbb{R}^d}}\\
&\h + \frac{1}{2}\int_0^{t\land \tau_{D_1}} \langle \II(\cV_{s-},J_{s-}),\II(dX^c_s,dX^c_s)\rangle_{\delta_{\mathbb{R}^d}}\\
&\h - \frac{1}{2}\int_0^{t\land \tau_{D_1}} \langle \II(\cV_{s-},dX^c_s), \II(J_{s-},dX^c_s) \rangle_{\delta_{\mathbb{R}^d}} \\
&\h + \sum_{0<s\leq t \land \tau_{D_1}} \langle \II( \cV_{s\land \tau_{D_1}-} ,J_{s \land \tau_{D_1}-}), \gamma^{\perp}(X_{s \land \tau_{D_1}-},X_{s\land \tau_{D_1}}) \rangle_{\delta_{\mathbb{R}^d}}
\end{align*}
is a local martingale. Rewriting the continuous bounded variation part of the right-hand side using
\begin{align*}
&\int_0^{t\land \tau_{D_1}} \langle \II(\cV_{s-},J_{s-}),\II(dX^c_s,dX^c_s)\rangle_{\delta_{\mathbb{R}^d}} - \int_0^{t\land \tau_{D_1}} \langle \II(\cV_{s-},dX^c_s), \II(J_{s-},dX^c_s) \rangle_{\delta_{\mathbb{R}^d}} \\
&=\int_0^{t\land \tau_{D_1}} g(\cV_s, R_M^{\nabla}(J_{s},dX^{c}_s)dX^{c}_s ),
\end{align*}
we deduce that the process \eqref{Difflocalmartingale} is a local martingale.
\end{proof}
\begin{rem}\label{truemart}
In addition to the assumptions of \tref{DiffHarmonic}, if both $E$ and $M$ are compact, $u$ is an $\cE$-harmonic map on the whole $E$, and $\bV_0$ is bounded, then the process \eqref{Difflocalmartingale} is a $\Prob_z$-martingale for q.e. $z\in E$ and we obtain
\begin{align}\label{ExpMart}
\Ex_z\left[g(\bV_0, \cH u(z))\right]=\Ex_z &\left[ g(\cV_{t\land \tau_{D_1}}, J_{t\land \tau_{D_1}})+\frac{1}{2}\int_0^{t\land \tau_{D_1}} g\left(\cV_s, R_M^{\nabla}\left( J_{s},dX^{\tau_{D_1},c}_s\right)dX^{\tau_{D_1},c}_s \right)\right.\nonumber \\
&\left.+ \sum_{0<s\leq t \land \tau_{D_1}} \langle \II( \cV_{s-} ,J_{s-}), \gamma^{\perp}(X_{s-},X_{s}) \rangle  \right].
\end{align}
In order to verify \eqref{ExpMart}, we show that $\dis Q_t:= \int^{\cdot \land \tau_{D_1}}\lambda_{g^{\bcl}}(h^{\nabla}_{\bJ}(\cV_-))\, d^{\gamma^{\bcl}}\bJ$ is a martingale. Since $E$ and $M$ are compact, both $u$ and $\cH u$ are componentwise in $\cF$. Thus the local martingale parts of $X$ and $J$ are square-integrable. In addition, since $u$ and $\cH u$ are in $L^{\infty}$, $J$ is a bounded process and $\II$ is bounded on the image of $u$. Moreover, since $\bV_0$ is bounded, $\cV$ is bounded by \pref{SDEtransport}. Since
\begin{align*}
h^{\nabla}_{\bJ}(\cV)=(\cV, \II(\cV,J)),
\end{align*}
it holds that
\[
Q_t=\int_0^{t\land \tau_{D_1}} \langle \cV_{s-}, dJ_s \rangle_{\delta_{\mathbb{R}^d}}+\int_0^{t\land \tau_{D_1}} \langle \II(\cV_{s-},J_{s-}), dX_s \rangle_{\delta_{\mathbb{R}^d}}.
\]
Therefore, there exists a constant $C>0$ such that
\[
[Q,Q]_t \leq C \left( [X,X]_{t\land \tau_{D_1}} + [J,J]_{t\land \tau_{D_1}} \right).
\]
Hence $Q$ is a square-integrable martingale.
\end{rem}
\begin{ex}
We assume $(\cE,\cF)$ has only the jump term and $M=\mathbb{S}^d$, the $d$-dimensional sphere. Then only the jump term of \eqref{Difflocalmartingale} remains and
\[
\langle \II( \cV_{s-} ,J_{s-}), \gamma^{\perp}(X_{s-},X_s) \rangle_{\delta_{\mathbb{R}^{d+1}}}=\frac{1}{2}|X_s-X_{s-}|^2g(\cV_{s-}, J_{s-})
\]
since $\II_x(\bU,\bV)=-g(\bU,\bV ) x$ and
\[
\gamma^{\perp}(x,y)=\langle x, y-x \rangle_{\mathbb{R}^{d+1}} x=-\frac{1}{2}|y-x|^2x.
\]
Thus for the jump part, using the L\'evy system formula \eqref{LevySys},
\begin{align*}
&\sum_{0<s\leq t \land \tau_{D_1}} \langle \II( \cV_{s-} ,J_{s-}), \gamma^{\perp}(X_{s-},X_{s}) \rangle_{\delta_{\mathbb{R}^{d+1}}}\\
&- \frac{1}{2}\int_0^{t \land \tau_{D_1}} \int_E |u(z)-u(Z_{s-})|^2N(Z_s,dz)g(\cV_{s-},J_{s-})\, dH_s 
\end{align*}
is a local martingale. We set
\begin{align*}
\rho_u(z)&:=\frac{1}{2}\int_E|u(w)-u(z)|^2\, N(z,dw),\\
Y_t&:=g(\cV_t,J_t).
\end{align*}
Then
\[
Y_{t\land \tau_{D_1}}-Y_0+\int_0^{t\land \tau_{D_1}} \rho_u(Z_s)Y_s\, dH_s
\]
is a local martingale. Thus if we set
\[
A_t=\exp \left( \int_0^t \rho_u(Z_s)\, dH_s \right),
\]
then the process $(AY)^{\tau_{D_1}}$ is a local martingale.
\end{ex}

\begin{thm}\label{diffLevy}
Let $(E,h)$ be a compact connected Riemannian manifold.
Let $(\cE,\cF)$ be the Dirichlet form on $E$ associated with an isotropic L\'evy process given in Subsection \ref{Levymfd} and assume that $(\cE,\cF)$ has a Brownian part with coefficient $\sigma>0$.
Let $\cH_i$ ($i=1,\dots l$) be Killing vector fields on $E$.
Assume that $u$ is $\cE$-harmonic on the whole space $E$ and satisfies the assumptions of \tref{DiffHarmonic} with $\cH=\cH_i$ for each $i=1,\dots,l$.
We set $X:=u(Z)$, $J_i:=\cH_iu(Z)$, $\bJ_i:=(X,J_i)$. Fix a bounded $\bV_0$ and let $\cV$ be the solution of \eqref{parallel}.
Fix $T>0$. For each $\mathbb{R}^l$-valued \cl $\{\cF_t^Z\}_{t\geq 0}$-adapted process $K$ with a deterministic initial value such that $K$ has almost surely absolutely continuous paths and $\int_0^T|\dot{K_t}|^2dt<\infty$, we set
\begin{align*}
I_t&:=\int_0^t \lambda_g(\cV_{s-})\, d^{\gamma}X_s,\\
L^K_t&:=\frac{1}{\sigma}\int_0^t \dot{K^i}_s\lambda_h(\cH_i)(Z_{s-})\, dZ^c_s,
\end{align*}
where $\lambda_h \colon TE \to T^*E$ is the bundle isomorphism associated with the metric $h$ and $dZ^c$ stands for the continuous local martingale part of $d^{\gamma'}Z$ for some connection rule $\gamma'$ on $E$, which does not depend on the choice of $\gamma'$.
We set $Y_t:=g(\cV_t, K^iJ_{i,t})$. Then for q.e. $z\in E$, the process
\begin{align}\label{IBPmart}
Y_t-Y_0- L^K_tI_t + \frac{1}{2} \int_0^tg(\cV_s, R_M^{\nabla}(K^i_s J_{i,s},dX^c_s)dX^c_s ) + \sum_{0<s\leq t}\langle \II( \cV_{s-} , K^i_s J_{i,s-}), \gamma^{\perp}(X_{s-},X_{s}) \rangle_{\delta_{\mathbb{R}^d}}
\end{align}
is a $(\Prob_z,\{\cF_t^Z\}_{t\geq 0})$-local martingale. Moreover, if $K$ satisfies
\begin{align}\label{Kbdd}
\Ex_z\left[ |K_0|^2+\int_0^T|\dot{K_s}|^2 ds \right]<\infty,
\end{align}
and $K_T=0$ a.s., then the process \eqref{IBPmart} is a true martingale and
\begin{align}\label{IBP}
\Ex_z\left[ g(\bV_0, K_0^i\cH_iu(z)) \right]&=\Ex_z \left[ -L^K_TI_T + \frac{1}{2} \int_0^T g(\cV_t, R_M^{\nabla}(K^i_t J_{i,t},dX^c_t)dX^c_t )\right. \nonumber\\
&\h \left.+ \sum_{0<t \leq T}\langle \II( \cV_{t-} , K^i_t J_{i,t-}), \gamma^{\perp}(X_{t-},X_{t}) \rangle_{\delta_{\mathbb{R}^d}} \right].
\end{align}
\end{thm}
\begin{proof}
By \tref{DiffHarmonic}, each $\bJ_i$ is a $(\Prob_z,\{\cF_t^Z\}_{t\geq 0})$-$\gamma^{\bcl}$-martingale for q.e. $z\in E$. We take such $z$.
Under the assumptions on $Z$, we have
\[
[I,L^K]_t=\int_0^t\langle \cV_s, \cH_iu(Z_s) \rangle \dot{K}^i_s\, ds.
\]
Indeed, by taking an isometric embedding $\iota \colon E \to \mathbb{R}^D$, we can rewrite
\[
L^K_t=\frac{1}{\sigma}\int_0^t \langle \dot{K}^i_{s}\iota_*\cH_i(Z_{s-}), dM^{[\iota],c}_s \rangle_{\mathbb{R}^D}.
\]
Since the continuous part of $(\cE,\cF)$ is
\[
\cE^c(\phi,\phi)=\frac{\sigma}{2} \int_E h(\nabla^E \phi, \nabla^E \phi)(z)\, dz
\]
for $\phi \in \cF$, where $\nabla^E$ is the gradient with respect to $h$, we have
\begin{align*}
\langle u^k(Z),M^{[\iota^j],c} \rangle_t&=\sigma \int_0^t h(\nabla^E u^k \cdot \nabla^E \iota^j) (Z_s)\, ds\\
&=\sigma \int_0^t \iota^j_*\nabla^E u^k(Z_s)\, ds,\ \Prob_z\text{-a.s. for q.e.}\ z\in E.
\end{align*}
Thus it holds that
\begin{align*}
[I,L^K]_t&=\frac{1}{\sigma}\sum_{k=1}^d\sum_{j=1}^D\int_0^t \cV_{k,s-} \dot{K}^i_{s}\iota^j_*\cH_i(Z_s) \, d\langle u^k(Z),M^{[\iota^j],c}\rangle\\
&=\sum_{k=1}^d\sum_{j=1}^D\int_0^t \cV_{k,s-}\dot{K}^i_{s}\iota^j_*\cH_i(Z_s) \iota^j_*\nabla^Eu^k(Z_s)\, ds\\
&=\int_0^t \langle \cV_s, \cH_iu(Z_s) \rangle_{\delta_{\mathbb{R}^d}} \dot{K}^i_s \, ds.
\end{align*}
Therefore
\[
d(L^KI)_t=L^K_t\, dI_t+I_{t-}\, dL^K_t + g(\cV_t, J_{i,t}) \dot{K}^i_t\, dt
\]
and the first two terms are local martingales.
Thus by integration by parts and \tref{MeanValue}, letting $\dis Q_i=\int \lambda_{g^{\bcl}}(h^{\nabla}_{\bJ_i}(\cV_-))\, d^{\gamma^{\bcl}}\bJ_i$, we obtain
\begin{align*}
dY_t&=K^i_t\, d \left(g( \cV_t, J_{i,t})\right) + g( \cV_{t-}, J_{i,t-}) \dot{K}^i_t\, dt\\
&= K^i_t\left(-\frac{1}{2} g(\cV_t, R_M^{\nabla}(J_{i,t},dX^c_t)dX^c_t ) - \langle \II( \cV_{t-}, J_{i,t-}), \gamma^{\perp}(X_{t-},X_{t}) \rangle_{\delta_{\mathbb{R}^d}}+dQ_{i,t}\right)\\
&\h + d\left(L^K_tI_t\right)-\left( L^K_t\, dI_t+I_{t-}\, dL^K_t \right).
\end{align*}
Thus the process \eqref{IBPmart} equals
\begin{align}\label{IBPmart2}
\int K^i \, dQ_i - \int L^K\, dI - \int I_-\, dL^K
\end{align}
and it is a local martingale.
Next, we assume that $K$ satisfies \eqref{Kbdd} and $K_T=0$ a.s. In the same way as in \rref{truemart}, noting that the local martingale parts of $X,J$ are square-integrable martingales, each $Q_i$ is also a square-integrable martingale. In addition, since $\cV_t$ is bounded by \pref{SDEtransport}, $I$ is also a square-integrable martingale. For $L^K$, since each $\cH_i$ is bounded, there exists a constant $C>0$ such that
\[
[L^K,L^K]_T \leq C\int_0^T|\dot{K_s}|^2\, ds.
\]
Thus by \eqref{Kbdd}, $L^K$ is also a square-integrable martingale.
We also have
\begin{align*}
\sup_{0\leq s \leq T}|K_s| &\leq |K_0| + \sup_{0\leq s \leq T}\left|\int_0^s\dot{K_r}\, dr \right| \\
&\leq |K_0| + \sqrt{T}\left( \int_0^T |\dot{K_s}|^2 \, ds \right)^{\frac{1}{2}}
\end{align*}
by the absolute continuity of $K$.
Therefore, by the Burkholder-Davis-Gundy inequality,
\begin{align*}
\Ex_z \left[ \left[ \int K^i\, dQ_i, \int K^j\, dQ_j \right]_T^{\frac{1}{2}} \right] &\leq C\Ex_z \left[ \sup_{0 \leq s \leq T}|K_s|^2 \right]^{\frac{1}{2}}\Ex_z \left[ [Q,Q]_T \right]^{\frac{1}{2}} <\infty, \\
\Ex_z \left[ \left[ \int L^K\, dI, \int L^K\, dI \right]_T^{\frac{1}{2}} \right] &\leq C\Ex_z \left[ \sup_{0 \leq s \leq T}|L^K_s|^2 \right]^{\frac{1}{2}}\Ex_z \left[ [I,I]_T \right]^{\frac{1}{2}} <\infty,\\
\Ex_z \left[ \left[ \int I_-\, dL^K, \int I_-\, dL^K \right]_T^{\frac{1}{2}} \right] &\leq C\Ex_z \left[ \sup_{0 \leq s \leq T}|I_s|^2 \right]^{\frac{1}{2}}\Ex_z \left[ [L^K,L^K]_T \right]^{\frac{1}{2}} <\infty
\end{align*}
for some constant $C>0$.
Therefore, the process \eqref{IBPmart} is a martingale. By taking the expectation of \eqref{IBPmart} at $t=T$, we get \eqref{IBP}.
\end{proof}

\section*{Acknowledgements}
The author is grateful to Prof. Marc Arnaudon for valuable discussion and comments.

\section*{Statements and Declarations}
{\bf Competing interests:} No conflict is related to this article, and the author has no relevant financial or non-financial interests to disclose.\\
{\bf Ethical Approval:} Not applicable to this article.\\
{\bf Funding:} This work was supported by JSPS KAKENHI Grant Numbers 24K22827 and 25K17264.\\
{\bf Data Availability Statement:} Data sharing is not applicable to this article as no datasets were generated or analyzed during the current study.

\end{document}